\documentclass[12pt]{amsart}

\usepackage{amsmath,amsthm,amsfonts,amssymb}
\usepackage[margin=1in]{geometry}    
\usepackage{parskip}  

\newtheorem{definition}{Definition}[section]
\newtheorem{theorem}{Theorem}[section]
\newtheorem{lemma}{Lemma}[section]

\newtheorem{prop}{Proposition}[section]
\newtheorem{remark}{Remark}[section]
\numberwithin{equation}{section}

\newcommand{\R}{\mathbb{R}}
\newcommand{\Z}{\mathbb{Z}}

\newcommand{\N}{\mathbb{N}}

\newcommand{\E}{\mathbb{E}}

\makeatletter
\@namedef{subjclassname@2020}{%
  \textup{2020} Mathematics Subject Classification}
\makeatother

\begin{document}

\title[SMALL BALL PROBABILITIES FOR SPDE WITH COLORED NOISE]{SMALL BALL PROBABILITIES FOR THE STOCHASTIC HEAT EQUATION WITH COLORED NOISE}
\author{Jiaming Chen}
\address{Dept. of Mathematics
\\University of Rochester
\\Rochester, NY  14627}
\email{jchen1994@urgrad.rochester.edu}
\begin{abstract}
We consider the stochastic heat equation on the 1-dimensional torus $\mathbb{T}:=\left[-1,1\right]$ with periodic boundary conditions:
\[
\partial_t u(t,x)=\partial^2_x u(t,x)+\sigma(t,x,u)\dot{F}(t,x),\quad x\in \mathbb{T},t\in\R_+,
\]
where $\dot{F}(t,x)$ is a generalized Gaussian noise, which is white in time and colored in space. Assuming that $\sigma$ is Lipschitz in $u$ and uniformly bounded, we estimate small ball probabilities for the solution $u$ when $u(0,x)\equiv 0$.
\end{abstract}

\maketitle

\section{Introduction}
In this paper, we estimate small ball probabilities for the solution of the stochastic heat equation of the type:
\begin{equation}
\label{SHE}
\partial_t u(t,x)=\partial^2_x u(t,x)+\sigma(t,x,u)\dot{F}(t,x)
\end{equation}
where $t\in\R_+$, $x\in\mathbb{T}:=\left[-1,1\right]$, which denotes the 1-dimensional torus in the usual way. The centered Gaussian noise $\dot{F}$ is white in time and homogeneous in space, i.e.,
\begin{equation*}
\mathbb{E}\left[\dot{F}(t,x)\dot{F}(s,y)\right]=\delta_0(t-s)\Lambda(x-y),
\end{equation*}
where $\delta_0$ is the Dirac delta generalized function and $\Lambda:\mathbb{T}\to\R_+$ is a non-negative generalized function whose Fourier series is given by
\begin{equation}
\label{lambdafourier}
\Lambda(x)=\sum_{n=0}^\infty q_n\exp(n\pi i x).
\end{equation}

In this article, we focus on an analogue of the Riesz kernel on $\R$. To be more precise, for $\gamma\in(0,1)$, there exist two positive constants $C$ and $C'$ such that
\begin{equation}
\label{qn}
q_n=\begin{cases}
C & \text{if~$n=0$};\\
C'\vert n\vert^{-(1-\gamma)}& \text{if~$n\neq 0$}.
\end{cases}
\end{equation}
Moreover, \cite{brosamler1983laws}, Lemma 2.9, proved that there exists a constant $C(\gamma)>0$ such that the spatial covariance function $\Lambda(x)$ admits the following estimate:
\begin{equation}
\label{Lambdabound}
\Lambda(x)\leq C(\gamma)\vert x\vert^{-\gamma},\quad\forall x\in\mathbb{T}.
\end{equation}

Small ball problems have a long history, and one can see \cite{li2001gaussian} for more surveys. In short, we are interested in the probability that a stochastic process $X_t$ starting at 0 stays in a small ball for a long time period, i.e.,
$$P\left(\sup_{0\leq t\leq T}\vert X_t\vert<\varepsilon\right),$$
where $\varepsilon>0$ is small. A recent paper \cite{athreya2021small} studied small ball probabilities for the solution to this type of stochastic heat equation driven by a space-time white noise and another paper \cite{foondun2022small} considered small ball probabilities for the solution to the same stochastic heat equation but under a H\"older semi-norm. The objective of this paper is to continue this topic with a white-in-time but colored-in-space noise.

Thanks to the Da Prato--Kwapien--Zabczyk's factorization and the assumption \eqref{assume1} and \eqref{assume2}, we may estimate the regularities of the stochastic term driven by a spatially homogeneous noise, which allows us to prove the critical tail bounds on the noise term as shown in Lemma \ref{larged}.

Regarding the temporal coefficient $c_0$ in \cite{athreya2021small}, it was demonstrated that it attains a uniform bound due to the exponential decay of the correlation between two points on the spatial interval as they separate. The correlation decays much slower in the case of spatially homogeneous noise. We therefore cannot bound $c_0$ uniformly, but we could control $c_0$ in terms of $\varepsilon$ indicated in \eqref{c0}.

\section{Main Result}
We write $u_0(x):=u(0,x)$ and assume that $\sigma$ is Lipschitz continuous in the third variable, i.e., there exists a constant $\mathcal{D}>0$ such that for all $t\geq 0$, $x\in\mathbb{T}$, $u,v\in\R$,
\begin{equation}
\label{assume1}
\vert \sigma(t,x,u)-\sigma(t,x,v)\vert\leq\mathcal{D}|u-v|,
\end{equation}
and $\sigma$ is uniformly bounded, i.e., there exist constants $\mathcal{C}_1$, $\mathcal{C}_2>0$ such that for all $t\geq 0$, $x\in\mathbb{T}$, $u\in\R$,
\begin{equation}
\label{assume2}
\mathcal{C}_1\leq \sigma(t,x,u)\leq \mathcal{C}_2.
\end{equation}
Assumption \eqref{assume2} assures that $\sigma$ is non-zero, which is helpful for us to control the variance of noise terms. In fact, \eqref{SHE} is not well-posed since the solution $u$ is not differentiable and $\dot{F}$ exists as a generalized function. However, the existence and uniqueness of the mild solution $u(t,x)$ to (1.1) under the above assumptions are well known as follow,
\begin{equation}
\label{mild}
\begin{split}
u(t,x)&=\int_{\mathbb{T}}G(t,x-y)u_0(y)dy\\
&\hspace{3cm}+\int_0^t\int_{\mathbb{T}}G(t-s,x-y)\sigma(s,y,u(s,y))F(dyds),
\end{split}
\end{equation}
where $G:\R_+\times\mathbb{T}\to \R$ is the fundamental solution of the heat equation on $\mathbb{T}$, 
\begin{align*}
\partial_t&G(t,x)=\partial_x^2 G(t,x),\\
&G(0,x)=\delta_0(x).
\end{align*}
See for example \cite{walsh1986anintroductiontostochastic}, \cite{dalang1999extending}, \cite{dalang2009minicourse}, and \cite{tindel1999space} for the proofs and various other properties. 
Now we are ready to state the main theorem.

\begin{theorem}
\label{thm}
Under the assumptions \eqref{assume1} and \eqref{assume2}, and a given parameter $\gamma$ in spatial covariance function defined in \eqref{lambdafourier} and \eqref{qn}, if $u(t,x)$ is the solution to \eqref{SHE} with $u_0(x)\equiv 0$, then there are positive constants $\textbf{C}_0,\textbf{C}_1,\textbf{C}_2,\textbf{C}_3,\mathcal{D}_0$ and $\varepsilon_0$ depending only on $\mathcal{C}_1,\mathcal{C}_2$ and $\gamma$, such that for all $\mathcal{D}<\mathcal{D}_0,0<\varepsilon<\varepsilon_0$ and $T>1$, we have
$$
\textbf{C}_0\exp\left(-\frac{\textbf{C}_1T}{\varepsilon^{\frac{8-2\gamma}{2-\gamma}}}\right)\leq P\left(\sup_{\substack{0\leq t\leq T\\ x\in\mathbb{T}}}\vert u(t,x)\vert\leq \varepsilon\right)\leq\textbf{C}_2\exp\left(-\frac{\textbf{C}_3T}{\varepsilon^{\frac{2\gamma+4}{2-\gamma}}}\right).
$$
\end{theorem}
Here we make a couple of open problems. These could be of independent interests.

\begin{enumerate}
\item[(a)] When $\gamma=1$, $\Lambda(x)=\delta_0(x)$ and $\dot{F}$ becomes the space-time white noise. One reference paper \cite{bezdek2016weak} has proved the weak convergence of a probability measure corresponding to the solution of the nonlinear stochastic heat equation with colored noise. Therefore the result from this paper can be converted to the white noise case in \cite{athreya2021small} as $\gamma\uparrow 1$. The open problem is the remaining part where $\gamma\downarrow 0$. When $\gamma=0$ and $x\neq 0$, $\Lambda(x)$ is a constant and $\dot{F}$ becomes a smooth noise. Can we also get this case by taking a limit as $\gamma\downarrow 0$? 

\item[(b)] When $\gamma=1$, both lower and upper bounds include a term $\varepsilon^{-6}$ in the exponent. However, in other cases, there exists a discrepancy between the exponents of $\varepsilon$ in both lower and upper bounds, due to the fact that $c_0$ depends on $\varepsilon$, as defined in \eqref{c0}. 

\item[(c)] One may extend this result to other white-in-time and colored-in-space noises whose correlation kernel differs from $\Lambda(x)$. However, the problem gets more complicated if the noise becomes fractional in time since we can not apply the Markov property, which plays an essential role in the proof. 
\end{enumerate}

For the upper bound of Theorem \ref{thm}, we use a perturbation argument to approximate the solution $u$ by a Gaussian random field in a small region, accompanied by the proof of the bounds of a Gaussian process. In addition, we apply the Gaussian correlation inequality and a change of measure argument similar to \cite{athreya2021small} to determine the lower bound for a Gaussian process. Furthermore, we demonstrate that the error in the approximation can be effectively controlled if the time interval where the coefficient is frozen is appropriately selected. We also make key use of the Markov property of \eqref{SHE} regarding the time parameter $t$. Thanks to this property, we can reduce our analysis to the behavior of this non-Gaussian solution in tiny time intervals.

Here is the organization of the rest of this paper. In section \ref{keyprop}, we present Proposition \ref{prop} and its connection to Theorem \ref{thm}. We provide some useful estimates in section \ref{estimates}, and the proof of Proposition \ref{prop} in section \ref{proofprop} completes the paper.

Throughout the entire article, $C$ and $C'$ denote positive constants whose values may vary from line to line. The dependence of constants on parameters will be denoted by mentioning the parameters in parentheses.

\section{Key Proposition}\label{keyprop}
We decompose $[-1,1]$ into subintervals of length $\varepsilon^2$ and divide $[0,T]$ into intervals of length $c_0\varepsilon^4$ where $c_0$ denotes the temporal coefficient satisfying 
\begin{equation}
\label{c0}
c_0=\mathcal{C}_0\varepsilon^{\frac{4-4\gamma}{\gamma}},
\end{equation}
where $0<\mathcal{C}_0<\mathcal{C}$ and $\mathcal{C}$ is specified in the proof of Lemma \ref{coeffbound}.

\begin{remark}
Unlike the space-time white noise case in \cite{athreya2021small}, $c_0$ needs to be selected depending on $\varepsilon$ due to the large correlation between pairs of points. $c_0$ does not appear in either the upper bound nor lower bound for small ball probabilities. 
\end{remark}

We define $t_i=ic_0\varepsilon^4,x_{j}=j\varepsilon^2$ where $i\in\N$ and $j\in\Z$, and
\begin{equation*}
n_1:=\min\left\lbrace n\in\Z:n\varepsilon^2>1\right\rbrace.
\end{equation*}
It is clear that $x_{n_1}>1$ and $x_{n_1-1}\leq 1$, therefore, by symmetry, the point $x_j$ lies in $\left[-1,1\right]$ if
\begin{equation}
\label{jbound}
-n_1+1\leq j\leq n_1-1.
\end{equation}

For $n\geq 0$, we define the following sequence of events to determine the upper bound in Theorem \ref{thm},
\begin{equation}
\label{Fn}
F_{n}=\left\lbrace\vert u(t_n,x_j)\vert\leq t_1^{\frac{2-\gamma}{4}}\text{~for all $-n_1+1\leq j\leq n_1-1$}\right\rbrace,
\end{equation}
which is the event that the magnitude of the solution is small (recall that $t_1=c_0\varepsilon^4$) at time $t_n$ for each grid point.
In addition, let $E_{-1}=\Omega$. For $n\geq 0$, we define the following sequence of events to determine the lower bound in Theorem \ref{thm},
\begin{equation}
\label{En}
E_{n}=\left\lbrace\vert u(t_{n+1},x)\vert\leq \frac{\mathcal{C}_3}{3}t_1^{\frac{2-\gamma}{4}}~\text{and}~\vert u(t,x)\vert\leq \mathcal{C}_3t_1^{\frac{2-\gamma}{4}}~\forall t\in[t_n,t_{n+1}],x\in[-1,1]\right\rbrace,
\end{equation}
which is the event that the magnitude of the solution is small at any time in $[t_n,t_{n+1}]$ for any points within $[-1,1]$ and even smaller at terminal time $t_{n+1}$. $\mathcal{C}_3$ is a positive constant such that
\begin{equation}
\label{c3}
\mathcal{C}_3>6\mathcal{C}_2\sqrt{\frac{\ln C_5}{C_6}},
\end{equation}
where $\mathcal{C}_2$ is the uniform bound of $\sigma$ in \eqref{assume2} and $C_5,C_6$ are positive constants defined in Lemma \ref{larged}.

\begin{prop}
\label{prop}
Consider the solution to \eqref{SHE} with $u_0(x)\equiv 0$. Then, there exist positive constants $\textbf{C}_4,\textbf{C}_5,\textbf{C}_6,\textbf{C}_7,\mathcal{D}_0$ and $\varepsilon_1$ depending solely on $\mathcal{C}_1$, $\mathcal{C}_2$ and $\gamma$ such that for any $0<\varepsilon<\varepsilon_1$ and $\mathcal{D}<\mathcal{D}_0$, we have
\begin{enumerate}
\item[(a)]
\begin{equation*}
P\left(F_{n}\bigg| \bigcap_{k=0}^{n-1}F_{k}\right)\leq  {\bf C_4}\exp\left(-\frac{{\bf C_5}}{t_1^{\gamma/2}}\right),
\end{equation*}
\item[(b)]
\begin{equation*}
P\left(E_{n}\bigg| \bigcap_{k=-1}^{n-1}E_{k}\right)\geq {\bf C_6}\exp\left(-\frac{{\bf C_7}}{t_1^{1-\gamma/2}}\right).
\end{equation*}
\end{enumerate}
\end{prop}
We then demonstrate how Theorem \ref{thm} results from Proposition \ref{prop}.

\textbf{Proof of Theorem \ref{thm}}: 
The event $F_{n}$ in $\eqref{Fn}$ deals with the behavior of $u(t,x)$ at time $t_n$, thus combining these events together indicates
\[F:=\bigcap_{n=0}^{\left\lfloor\frac{T}{t_1}\right\rfloor}F_{n}\supset \left\lbrace\vert u(t,x)\vert\leq t_1^{\frac{2-\gamma}{4}}, t\in[0,T],x\in\mathbb{T}\right\rbrace,\]
and
\[P(F)=P\left(\bigcap_{n=0}^{\left\lfloor\frac{T}{t_1}\right\rfloor}F_{n}\right)=P(F_{0})\prod_{n=1}^{\left\lfloor\frac{T}{t_1}\right\rfloor}P\left(F_{n}\bigg| \bigcap_{k=0}^{n-1}F_{k}\right).\]

With $u_0(x)\equiv 0$, $F_0=\Omega$, and Proposition \ref{prop}(a) immediately yields 
\begin{align*}
P(F)&\leq \left[\textbf{C}_4\exp\left(-\frac{\textbf{C}_5}{t_1^{\gamma/2}}\right)\right]^{\left\lfloor\frac{T}{t_1}\right\rfloor}\leq \textbf{C}_2\exp\left(-\frac{\textbf{C}_3T}{t_1^{1+\gamma/2}}\right).
\end{align*}
Therefore we have
\[
P\left(\left\lbrace\vert u(t,x)\vert\leq t_1^\frac{2-\gamma}{4}, t\in[0,T],x\in\mathbb{T}\right\rbrace\right)\leq\textbf{C}_2\exp\left(-\frac{\textbf{C}_3T}{t_1^{1+\gamma/2}}\right),
\]
and replacing $t_1^\frac{2-\gamma}{4}$ with $\varepsilon_2$ gives
\begin{equation}
\label{upper2}
P\left(\left\lbrace\vert u(t,x)\vert\leq \varepsilon_2, t\in[0,T],x\in\mathbb{T}\right\rbrace\right)\leq\textbf{C}_2\exp\left(-\frac{\textbf{C}_3T}{\varepsilon_2^{\frac{4+2\gamma}{2-\gamma}}}\right).
\end{equation}
We derive that, from \eqref{c0},
\[
\varepsilon_2=t_1^\frac{2-\gamma}{4}=\left(c_0\varepsilon^4\right)^\frac{2-\gamma}{4}=\left(\mathcal{C}_0\varepsilon^{\frac{4}{\gamma}}\right)^\frac{2-\gamma}{4}<\left(\mathcal{C}_0\varepsilon_1^{\frac{4}{\gamma}}\right)^\frac{2-\gamma}{4},
\]
so \eqref{upper2} provides the upper bound in Theorem \ref{thm} whenever $\varepsilon_0<\left(\mathcal{C}_0\varepsilon_1^{\frac{4}{\gamma}}\right)^\frac{2-\gamma}{4}$. 

On the other hand, the event $E_{n}$ in $\eqref{En}$ deals with the behavior of $u(t,x)$ at any time between $t_n$ and $t_{n+1}$, thus combining these events together indicates
\[E:=\bigcap_{n=-1}^{\left\lfloor\frac{T}{t_1}\right\rfloor}E_{n}\subset \left\lbrace\vert u(t,x)\vert\leq \mathcal{C}_3t_1^{\frac{2-\gamma}{4}}, t\in[0,T],x\in\mathbb{T}\right\rbrace,\]
and
\[P(E)=P\left(\bigcap_{n=-1}^{\left\lfloor\frac{T}{t_1}\right\rfloor}E_{n}\right)=P(E_{-1})\prod_{n=0}^{\left\lfloor\frac{T}{t_1}\right\rfloor}P\left(E_{n}\bigg| \bigcap_{k=-1}^{n-1}E_{k}\right).\]
With $u_0(x)\equiv 0$ and $E_{-1}=\Omega$, Proposition \ref{prop}(b) immediately yields
\[P(E)\geq \left[\textbf{C}_6\exp\left(-\frac{\textbf{C}_7}{t_1^{1-\gamma/2}}\right)\right]^{\left\lfloor\frac{T}{t_1}\right\rfloor+1}\geq \textbf{C}_0\exp\left(-\frac{\textbf{C}_1T}{t_1^{2-\gamma/2}}\right).\]
Therefore we have 
\[
P\left(\left\lbrace\vert u(t,x)\vert\leq \mathcal{C}_3t_1^{\frac{2-\gamma}{4}}, t\in[0,T],x\in\mathbb{T}\right\rbrace\right)\geq\textbf{C}_0\exp\left(-\frac{\textbf{C}_1T}{t_1^{2-\gamma/2}}\right),
\]
and replacing $\mathcal{C}_3t_1^\frac{2-\gamma}{4}$ with $\varepsilon_3$ gives
\begin{equation}
\label{lower2}
P\left(\left\lbrace\vert u(t,x)\vert\leq \varepsilon_3, t\in[0,T],x\in\mathbb{T}\right\rbrace\right)\geq\textbf{C}_0\exp\left(-\frac{\textbf{C}_1T}{\varepsilon_3^{\frac{8-2\gamma}{(2-\gamma)}}}\right).
\end{equation}
We derive that, from \eqref{c0},
\[
\varepsilon_3=\mathcal{C}_3t_1^\frac{2-\gamma}{4}=\mathcal{C}_3\left(c_0\varepsilon^4\right)^\frac{2-\gamma}{4}=\mathcal{C}_3\left(\mathcal{C}_0\varepsilon^{\frac{4}{\gamma}}\right)^\frac{2-\gamma}{4}<\mathcal{C}_3\left(\mathcal{C}_0\varepsilon_1^{\frac{4}{\gamma}}\right)^\frac{2-\gamma}{4},
\]
so \eqref{lower2} provides the lower bound in Theorem \ref{thm} whenever $\varepsilon_0<\mathcal{C}_3\left(\mathcal{C}_0\varepsilon_1^{\frac{4}{\gamma}}\right)^\frac{2-\gamma}{4}$. 

With such an $\varepsilon_0<\left(\mathcal{C}_0\varepsilon_1^{\frac{4}{\gamma}}\right)^\frac{2-\gamma}{4}\cdot\min\left\lbrace\mathcal{C}_3,1\right\rbrace$, we conclude Theorem \ref{thm} from Proposition \ref{prop}, and the rest of this paper is devoted to the proof of Proposition \ref{prop}.

\section{Preliminaries}\label{estimates}
In this section, we provide some preliminary results that are used to prove Proposition \ref{prop}.
\subsection{Heat Kernel Estimates} In this subsection, we prove some results related to the heat kernel $G(t,x)$. To simplify our computation, we write $G_t(x):=G(t,x)$ and recall that the explicit form for the heat kernel on $\mathbb{T}$ is given by
$$
G_{t}(x)=\sum_{n\in\Z}(4\pi t)^{-1/2}\exp\left(-\frac{(x+2n)^2}{4t}\right).
$$
The following statement has been proven for $x\in\R$ in the paper \cite{sanz2002holder}. Here we prove for the case when $x\in\mathbb{T}$.

\begin{lemma} \label{heatlemma}There exist positive constants $C(\alpha,\xi)$ and $C(\alpha,\zeta)$ such that for all $\alpha\in(0,1)$, $\xi\in(0,2\alpha\wedge 1)$, $\zeta\in(0,\alpha)$, $0<s<t\leq 1$ and $x,y\in\mathbb{T}$, we have
\begin{equation}
\label{heatspace}
\int_0^t\int_{-1}^1\vert G_{t-r}(x-z)-G_{t-r}(y-z)\vert(t-r)^{\alpha-1} dzdr\leq C(\alpha,\xi)\vert x-y\vert^\xi,
\end{equation}
\begin{equation}
\label{heattime}
\int_0^s\int_{-1}^1\left\vert G_{t-r}(x-z)(t-r)^{\alpha-1}-G_{s-r}(x-z)(s-r)^{\alpha-1}\right\vert dzdr\leq C(\alpha,\zeta)\vert t-s\vert^\zeta.
\end{equation}
\end{lemma}

\begin{proof}
For any $\xi\in(0,2\alpha\wedge 1)$, we have
\begin{align*}
&\int_{-1}^1\vert G_{t-r}(x-z)-G_{t-r}(y-z)\vert(t-r)^{\alpha-1} dz\\
&\leq 2^{1-\xi}(t-r)^{\alpha-1-\frac{\xi}{2}}\left(\int_{-1}^1\left\vert\sum_{n\in\Z}\left[\exp\left(-\frac{(x+2n-z)^2}{4(t-r)}\right)\right.\right.\right.\\
&\hspace{7cm}\left.\left.\left.-\exp\left(-\frac{(y+2n-z)^2}{4(t-r)}\right)\right]\right\vert dz\right)^\xi\\
&\leq 2^{1-\xi}(t-r)^{\alpha-1-\frac{\xi}{2}}\left(\sum_{n\in\Z}\int_{-1}^1\left\vert\exp\left(-\frac{(x+2n-z)^2}{4(t-r)}\right)\right.\right.\\
&\hspace{7cm}\left.\left.-\exp\left(-\frac{(y+2n-z)^2}{4(t-r)}\right)\right\vert dz\right)^\xi\\
&=2^{1-\xi}(t-r)^{\alpha-1-\frac{\xi}{2}}\left(\int_\R\left\vert\exp\left(-\frac{(x-z)^2}{4(t-r)}\right)-\exp\left(-\frac{(y-z)^2}{4(t-r)}\right)\right\vert dz\right)^\xi\\
&\leq C(\xi)(t-r)^{\alpha-1-\frac{\xi}{2}}\vert x-y\vert^\xi.
\end{align*}
The first inequality follows that the integral $\int_{-1}^1\vert G_{t-r}(x-z)-G_{t-r}(y-z)\vert dz $ is bounded above by $2$. The last inequality is obtained by applying the Mean Value theorem to the function $x\to \exp\left(-\frac{(x-z)^2}{4(t-r)}\right)$, so there exists $x'\in(x,y)$ such that
$$\left\vert\exp\left(-\frac{(x-z)^2}{4(t-r)}\right)-\exp\left(-\frac{(y-z)^2}{4(t-r)}\right)\right\vert=\frac{1}{2}\vert x-y\vert\exp\left(-\frac{(x'-z)^2}{4(t-r)}\right)\cdot\frac{\vert x'-z\vert}{(t-r)}.$$ 
Because $x\leq \exp\left(\frac{x^2}{2}\right)$ for $\forall x\in\mathbb{R}$, we have
$$\frac{\vert x'-z\vert}{2(t-r)}=\frac{\vert x'-z\vert}{2(t-r)^{\frac{1}{2}}}\cdot\frac{1}{(t-r)^{\frac{1}{2}}}\leq \exp\left(\frac{(x'-z)^2}{8(t-r)}\right)\cdot\frac{1}{(t-r)^{\frac{1}{2}}},$$
which implies
\begin{equation*}
\begin{split}
&\int_\R\left\vert\exp\left(-\frac{(x-z)^2}{4(t-r)}\right)-\exp\left(-\frac{(y-z)^2}{4(t-r)}\right)\right\vert dz\\
&\hspace{4cm}=\int_\R\vert x-y\vert\exp\left(-\frac{(x'-z)^2}{4(t-r)}\right)\cdot\frac{\vert x'-z\vert}{2(t-r)}dz\\
&\hspace{4cm}\leq \vert x-y\vert\int_\R\exp\left(-\frac{(x'-z)^2}{8(t-r)}\right)\cdot\frac{1}{(t-r)^{\frac{1}{2}}} dz=C\vert x-y\vert.
\end{split}
\end{equation*}
Therefore, we derive
\begin{multline*}
\int_0^t\int_{-1}^1\vert G_{t-r}(x-z)-G_{t-r}(y-z)\vert(t-r)^{\alpha-1} dzdr\\
\leq \int_0^t C(\xi)(t-r)^{\alpha-1-\frac{\xi}{2}}\vert x-y\vert^\xi dr\leq C(\alpha,\xi)t^{\alpha-\frac{\xi}{2}}\vert x-y\vert^\xi\leq C(\alpha,\xi)\vert x-y\vert^\xi , 
\end{multline*}
which completes the proof of \eqref{heatspace}. 

For \eqref{heattime}, we decompose the integral into two parts,
\begin{equation}
\label{twoparts}
\begin{split}
&\int_0^s\int_{-1}^1\left\vert \sum_{n\in\Z}\left[\exp\left(-\frac{(x+2n-z)^2}{4(t-r)}\right)(t-r)^{\alpha-1-\frac{1}{2}}\right.\right.\\
&\hspace{6cm}\left.\left.-\exp\left(-\frac{(x+2n-z)^2}{4(s-r)}\right)(s-r)^{\alpha-1-\frac{1}{2}}\right]\right\vert dzdr\\
&\leq \int_0^s\int_\R\left\vert \exp\left(-\frac{(x-z)^2}{4(t-r)}\right)(t-r)^{\alpha-1-\frac{1}{2}}-\exp\left(-\frac{(x-z)^2}{4(s-r)}\right)(s-r)^{\alpha-1-\frac{1}{2}}\right\vert dzdr\\
&\leq \int_0^s\int_\R \left\vert \exp\left(-\frac{(x-z)^2}{4(t-r)}\right)(t-r)^{\alpha-1-\frac{1}{2}}-\exp\left(-\frac{(x-z)^2}{4(s-r)}\right)(t-r)^{\alpha-1-\frac{1}{2}}\right\vert dzdr\\
&+\int_0^s\int_\R \left\vert \exp\left(-\frac{(x-z)^2}{4(s-r)}\right)(t-r)^{\alpha-1-\frac{1}{2}}-\exp\left(-\frac{(x-z)^2}{4(s-r)}\right)(s-r)^{\alpha-1-\frac{1}{2}}\right\vert dzdr.
\end{split}
\end{equation}
For any $\zeta\in(0,\alpha)$, we have
\begin{equation}
\label{ineqzeta}
\begin{split}
&\left\vert \exp\left(-\frac{(x-z)^2}{4(t-r)}\right)-\exp\left(-\frac{(x-z)^2}{4(s-r)}\right)\right\vert\\
&\hspace{5cm}\leq\left\vert \exp\left(-\frac{(x-z)^2}{4(t-r)}\right)-\exp\left(-\frac{(x-z)^2}{4(s-r)}\right)\right\vert^\zeta.
\end{split}
\end{equation}
Apply the Mean Value theorem to function $t\to \exp\left(-\frac{(x-z)^2}{4(t-r)}\right)$, and there exists $t'\in(s,t)$ such that
\begin{align*}
&\left\vert\exp\left(-\frac{(x-z)^2}{4(t-r)}\right)-\exp\left(-\frac{(x-z)^2}{4(s-r)}\right)\right\vert=(t-s)\exp\left(-\frac{(x-z)^2}{4(t'-r)}\right)\cdot\frac{(x-z)^2}{4(t'-r)^2}\\
&\hspace{5cm}=2(t-s)(t'-r)^{-1}\exp\left(-\frac{(x-z)^2}{4(t'-r)}\right)\cdot\frac{(x-z)^2}{8(t'-r)}\\
&\hspace{5cm}\leq 2(t-s)(t'-r)^{-1}\exp\left(-\frac{(x-z)^2}{8(t'-r)}\right).
\end{align*}
The last inequality follows from the property $x\leq\exp(x)$ for $\forall x\in\R$. Plug the above into the RHS of \eqref{ineqzeta} to parametrize the upper bound in term of $\zeta$, thus we have 
$$\left\vert \exp\left(-\frac{(x-z)^2}{4(t-r)}\right)-\exp\left(-\frac{(x-z)^2}{4(s-r)}\right)\right\vert\leq C(\zeta)(t-s)^\zeta (s-r)^{-\zeta}\exp\left(-\zeta\frac{(x-z)^2}{8(t-r)}\right).$$
As a result,
\begin{equation}
\label{I1}
\begin{split}
&\int_0^s\int_\R \left\vert \exp\left(-\frac{(x-z)^2}{4(t-r)}\right)(t-r)^{\alpha-1-\frac{1}{2}}-\exp\left(-\frac{(x-z)^2}{4(s-r)}\right)(t-r)^{\alpha-1-\frac{1}{2}}\right\vert dzdr\\
&\leq C(\zeta)(t-s)^\zeta\int_0^s(s-r)^{-\zeta}(t-r)^{\alpha-1}\int_\R\exp\left(-\zeta\frac{(x-z)^2}{8(t-r)}\right)(t-r)^{-\frac{1}{2}}dzdr\\
&\leq C(\zeta)(t-s)^\zeta\int_0^s(s-r)^{\alpha-1-\zeta}dr\leq C_{\alpha,\zeta}s^{\alpha-\zeta}(t-s)^\zeta\leq C_{\alpha,\zeta}(t-s)^\zeta.
\end{split}
\end{equation}
Similarly, for any $\zeta\in(0,\alpha),r<s\leq t$, we have
$$\left\vert (t-r)^{\alpha-1-\frac{1}{2}}-(s-r)^{\alpha-1-\frac{1}{2}}\right\vert^{1-\zeta}\leq 2(s-r)^{(\alpha-1-\frac{1}{2})(1-\zeta)}.$$
If we multiply $\left\vert (t-r)^{\alpha-1-\frac{1}{2}}-(s-r)^{\alpha-1-\frac{1}{2}}\right\vert^{\zeta}$ to both sides, we obtain
\begin{equation}
\label{ineqtime}
\left\vert (t-r)^{\alpha-1-\frac{1}{2}}-(s-r)^{\alpha-1-\frac{1}{2}}\right\vert\leq 2(s-r)^{(\alpha-1-\frac{1}{2})(1-\zeta)}\left\vert (t-r)^{\alpha-1-\frac{1}{2}}-(s-r)^{\alpha-1-\frac{1}{2}}\right\vert^{\zeta}.
\end{equation}
Applying the Mean Value theorem to the function of $x\to (x-r)^{\alpha-1-\frac{1}{2}}$ yields
$$\left\vert (t-r)^{\alpha-1-\frac{1}{2}}-(s-r)^{\alpha-1-\frac{1}{2}}\right\vert = C(\alpha)(t-s)(t'-r)^{\alpha-2-\frac{1}{2}}$$
for some $t'\in(s,t)$. Plug the above into the RHS of \eqref{ineqtime} to parametrize the upper bound in term of $\zeta$, thus we have 
\begin{align*}
\left\vert (t-r)^{\alpha-1-\frac{1}{2}}-(s-r)^{\alpha-1-\frac{1}{2}}\right\vert&\leq C(\alpha,\zeta)(s-r)^{(\alpha-1-\frac{1}{2})(1-\zeta)+(\alpha-2-\frac{1}{2})\zeta}(t-s)^\zeta\\
&=C(\alpha,\zeta)(s-r)^{(\alpha-1-\frac{1}{2}-\zeta)}(t-s)^\zeta.
\end{align*}
Consequently,
\begin{equation}
\label{I2}
\begin{split}
&\int_0^s\int_\R \left\vert \exp\left(-\frac{(x-z)^2}{4(s-r)}\right)(t-r)^{\alpha-1-\frac{1}{2}}-\exp\left(-\frac{(x-z)^2}{4(s-r)}\right)(s-r)^{\alpha-1-\frac{1}{2}}\right\vert dzdr\\
&\leq C(\alpha,\zeta)(t-s)^\zeta\int_0^s(s-r)^{(\alpha-1-\zeta)}\int_\R (s-r)^{-\frac{1}{2}}\exp\left(-\frac{(x-z)^2}{4(s-r)}\right)dzdr\\
&\leq C(\alpha,\zeta)s^{\alpha-\zeta}(t-s)^\zeta\leq C(\alpha,\zeta)(t-s)^\zeta.
\end{split}
\end{equation}
Combining \eqref{twoparts}, \eqref{I1} and \eqref{I2} completes the proof of \eqref{heattime}.
\end{proof}

\subsection{Noise Term Estimates}
Define the second integral of \eqref{mild}, i.e. noise term, by 
\begin{equation}
\label{noiseterm}
N(t,x):=\int_0^t\int_{-1}^1G_{t-s}(x-y)\sigma(s,y,u(s,y))F(dyds).
\end{equation}
When $\sigma\equiv 1$, we may express $N(t,x)$ in terms of random Fourier series decomposition as
\begin{equation}
\label{noisetermexp}
\begin{split}
&N(t,x)=\sum_{n=0}^\infty\sqrt{q_n}\left(\cos(\pi nx)\int_0^te^{-\pi^2n^2(t-s)}\beta_n(ds)\right.\\
&\hspace{5cm}\left.+\sin(\pi nx)\int_0^te^{-\pi^2n^2(t-s)}\tilde{\beta}_n(ds)\right),
\end{split}
\end{equation}
where the sequence $\{\beta_n\}_{n=0}^\infty$ and $\{\tilde{\beta}_n\}_{n=0}^\infty$ are independent and each formed of independent 1-dimensional real-valued $(\mathcal{F}_t)$-adapted Brownian motions. \cite{tindel2003stochastic} Corollary 1 shows that the necessary and sufficient condition for existence of \eqref{noisetermexp} in $L^2(\Omega\times[0,T]\times[-1,1])$ is
$$\sum_{n=1}^\infty \frac{q_n}{n^2}<\infty.$$

We could apply a factorization method (see \cite{da2014stochastic} Theorem 5.10) to $N(t,x)$ under assumptions \eqref{assume1} and \eqref{assume2} of $\sigma$. By the semigroup property of $G$ and a stochastic Fubini's theorem (see \cite{walsh1986anintroductiontostochastic} Theorem 2.6), for all $\alpha\in(0,1)$, we have
$$N(t,x)=\frac{\sin(\pi\alpha)}{\pi}\int_0^t\int_{-1}^1G_{t-r}(x-z)(t-r)^{\alpha-1}Y_\alpha(r,z)dzdr,$$
where
$$Y_\alpha(r,z)=\int_0^r\int_{-1}^1G_{r-s}(z-y)\sigma(s,y,u(s,y))(r-s)^{-\alpha}F(dyds).$$
Thanks to this factorization, we could estimate the noise term. Before that, we prove the following lemma which is similar to \cite{sanz2002holder} Lemma 2.2.

\begin{lemma}\label{Ylemma} For any $T\geq 0$, $\alpha\in\left(0,\frac{2-\gamma}{4}\right)$,
\begin{equation*}
\sup_{0\leq r\leq T}\sup_{z\in \mathbb{T}}\mathbb{E}\left[Y_\alpha^2(r,z)\right]<+\infty.
\end{equation*}
\end{lemma}
\begin{proof} By assumption \eqref{assume2} of $\sigma$, we have
\begin{align*}
&\mathbb{E}\left[Y_\alpha^2(r,z)\right]=\mathbb{E}\left[\left(\int_0^r\int_{-1}^1G_{r-s}(z-y)\sigma(s,y,u(s,y))(r-s)^{-\alpha}F(dyds)\right)^2\right]\\
&\leq \sup_{s,y}\mathbb{E}\left[\sigma^2(s,y,u(s,y))\right]\cdot\mathbb{E}\left[\left(\int_0^r\int_{-1}^1G_{r-s}(z-y)(r-s)^{-\alpha}F(dyds)\right)^2\right]\\
&\leq \mathcal{C}_2^2\mathbb{E}\left[\left(\sum_{n=0}^\infty\sqrt{q_n}\left(\cos(n\pi z)\int_0^r(r-s)^{-\alpha} e^{-\pi^2n^2(r-s)}\beta_n(ds)\right.\right.\right.\\
&\hspace{5cm}\left.\left.\left.+\sin(n\pi z)\int_0^r (r-s)^{-\alpha}e^{-\pi^2n^2(r-s)}\tilde{\beta}_n(ds)\right)\right)^2\right]\\
&=\mathcal{C}_2^2\sum_{n=0}^\infty q_n\left(\sin^2(n\pi z)+\cos^2(n\pi z)\right)\int_0^r(r-s)^{-2\alpha} e^{-2\pi^2n^2(r-s)}ds\\
&=\mathcal{C}_2^2\sum_{n=0}^\infty q_n\int_0^r(r-s)^{-2\alpha} e^{-2\pi^2n^2(r-s)}ds\\
&\leq C\mathcal{C}_2^2\left[q_0+\sum_{n=1}^\infty q_n n^{4\alpha-2}\Gamma(1-2\alpha)\right]\\
&<C+C'\sum_{n=1}^\infty n^{4\alpha+\gamma-3}<\infty
\end{align*}
The second inequality follows the Fourier series decomposition of a spatial convolution, and we substitute $2\pi^2n^2(r-s)$ in the last inequality to get an Euler gamma function. Recall $q_n$ in \eqref{qn}, and the summation converges so that the lemma follows.
\end{proof}

We will now estimate the spatial regularity of $N(t,x)$ with $t\leq 1$,
\begin{align*}
&\mathbb{E}\left[(N(t,x)-N(t,y))^2\right]\\
&=C(\alpha) \mathbb{E}\left[\left(\int_0^t\int_{-1}^1[G_{t-r}(x-z)-G_{t-r}(y-z)](t-r)^{\alpha-1}Y_\alpha(r,z)dzdr\right)^2\right]
\\
&\leq C(\alpha) \mathbb{E}\left[\left(\int_0^t\int_{-1}^1\left\vert[G_{t-r}(x-z)-G_{t-r}(y-z)](t-r)^{\alpha-1}Y_\alpha(r,z)\right\vert dzdr\right)^2\right]
\\
&\leq C(\alpha) \sup_{r,z}\mathbb{E}\left[Y_\alpha^2(r,z)\right]\left(\int_0^t\int_{-1}^1\left\vert [G_{t-r}(x-z)-G_{t-r}(y-z)](t-r)^{\alpha-1}\right\vert dzdr\right)^2.
\end{align*}
By Lemma \ref{heatlemma} and Lemma \ref{Ylemma}, we conclude that there exists a constant $C>0$ depending only on $\alpha,\gamma$ and $\xi$ with $\xi\in(0,2\alpha)$ and $\alpha\in\left(0,\frac{2-\gamma}{4}\right)$ such that
\begin{equation}
\label{spacereg}
\mathbb{E}\left[(N(t,x)-N(t,y))^2\right]\leq C\mathcal{C}_2^2\vert x-y\vert^{2\xi}.
\end{equation}

For temporal regularity,
\begin{align*}
&\mathbb{E}\left[(N(t,x)-N(s,x))^2\right]\\
&= C(\alpha) \mathbb{E}\left[\left(\int_0^s\int_{-1}^1\left[G_{t-r}(x-z)(t-r)^{\alpha-1}-G_{s-r}(x-z)(s-r)^{\alpha-1}\right]Y_\alpha(r,z)dzdr\right.\right.\\
&\hspace{5cm}\left.\left.+\int_s^t\int_{-1}^1G_{t-r}(x-z)(t-r)^{\alpha-1}Y_{\alpha}(r,z)dzdr\right)^2\right]\\
&= C(\alpha) \left(\mathbb{E}\left[\left( \int_0^s\int_{-1}^1\left[G_{t-r}(x-z)(t-r)^{\alpha-1}-G_{s-r}(x-z)(s-r)^{\alpha-1}\right]Y_\alpha(r,z)dzdr\right)^2\right]\right.\\
&\hspace{5cm}\left.+\mathbb{E}\left[\left( \int_s^t\int_{-1}^1G_{t-r}(x-z)(t-r)^{\alpha-1}Y_{\alpha}(r,z)dzdr\right)^2\right]\right)\\
&\leq C(\alpha)\sup_{r,z}\mathbb{E}\left[Y_\alpha^2(r,z)\right]\left(I_1^2+I_2^2\right),
\end{align*}
where
\begin{equation*}
I_1=\int_0^s\int_{-1}^1\left\vert G_{t-r}(x-z)(t-r)^{\alpha-1}-G_{s-r}(x-z)(s-r)^{\alpha-1}\right\vert dzdr,
\end{equation*}
and
\begin{equation}
\label{I2equ}
I_2=\int_s^t\int_{-1}^1\left\vert G_{t-r}(x-z)(t-r)^{\alpha-1}\right\vert dzdr=C_\alpha(t-s)^\alpha<C_\alpha(t-s)^\zeta.
\end{equation}
By Lemma \ref{heatlemma} and for $0\leq s<t\leq 1$, we derive
\begin{equation}
\label{I1equ}
I_1\leq C_{\alpha,\zeta}(t-s)^\zeta.
\end{equation}
Therefore, we conclude that there exists a constant $C>0$ depending only on $\alpha,\gamma$ and $\zeta$ with $\zeta\in(0,\alpha)$ and $\alpha\in\left(0,\frac{2-\gamma}{4}\right)$ such that
\begin{equation*}
\label{timereg}
\mathbb{E}\left[(N(t,x)-N(s,x))^2\right]\leq C\mathcal{C}_2^2(t-s)^{2\zeta}.
\end{equation*}

\begin{lemma}There exist positive constants $C_1,C_2$ depending only on $\alpha,\gamma$ and $\xi$, and positive constants $C_3,C_4$ depending only on $\alpha,\gamma$ and $\zeta$ such that for all $0\leq s<t\leq 1$, $x,y\in\mathbb{T}$, $\alpha\in\left(0,\frac{2-\gamma}{4}\right)$, $\xi\in(0,2\alpha),\zeta\in(0,\alpha)$, and $\lambda>0$, we have
\begin{equation}
\label{spacep}
P(\vert N(t,x)-N(t,y)\vert>\lambda)\leq C_1\exp\left(-\frac{C_2\lambda^2}{\mathcal{C}_2^2\vert x-y\vert^{2\xi}}\right),
\end{equation}
\begin{equation}
\label{timep}
P(\vert N(t,x)-N(s,x)\vert>\lambda)\leq C_3\exp\left(-\frac{C_4\lambda^2}{\mathcal{C}_2^2\vert t-s\vert^{2\zeta}}\right).
\end{equation}
\end{lemma}
\begin{proof}
Define 
$$N_t(s,x):=\int_0^s\int_{-1}^1 G_{t-r}(x-y)\sigma(r,y,u(r,y))F(dydr).$$
Note that $N_t(t,x)=N(t,x)$ and $N_t(s,x)$ is a continuous $\mathcal{F}_s^F$-adapted martingale in $s$ since the integrand does not depend on $s$. Consider
$$M_s:=N_t(s,x)-N_t(s,y)=\int_0^s\int_{-1}^1 [G_{t-r}(x-z)-G_{t-r}(y-z)]\sigma(r,z,u(r,z))F(dzdr),$$
so that $M_t=N(t,x)-N(t,y)$. As $M_s$ is a continuous $\mathcal{F}_s^F$-adapted martingale with $M_0=0$, it is a time changed Brownian motion. In particular, we have
$$M_t=B_{\langle M\rangle_t},$$
and \eqref{spacereg} gives an upper bound on the time change as
$$\langle M\rangle_t\leq C\mathcal{C}_2^2\vert x-y\vert^{2\xi}.$$
Therefore, by the reflection principle for the Brownian motion $B_{\langle M\rangle_t}$,
\begin{align*}
P(N(t,x)-N(t,y)>\lambda)&=P(M_t>\lambda)=P\left(B_{\langle M\rangle_t}>\lambda\right)\\
&\leq P\left(\sup_{s\leq C\mathcal{C}_2^2\vert x-y\vert^{2\xi}}B_s>\lambda\right)=2P\left(B_{C\mathcal{C}_2^2\vert x-y\vert^{2\xi}}>\lambda\right)\\
&\leq C_1\exp\left(-\frac{C_2\lambda^2}{\mathcal{C}_2^2\vert x-y\vert^{2\xi}}\right).
\end{align*}
Switching $x$ and $y$ gives
\begin{align*}
P(-N(t,x)+N(t,y)>\lambda)&=P(M_t<-\lambda)\leq 2P\left(B_{C\mathcal{C}_2^2\vert x-y\vert^{2\xi}}<-\lambda\right)\\
&\leq C_1\exp\left(-\frac{C_2\lambda^2}{\mathcal{C}_2^2\vert x-y\vert^{2\xi}}\right).
\end{align*}
Consequently, for $\forall\xi\in(0,2\alpha)$, 
$$
P(\vert N(t,x)-N(t,y)\vert>\lambda)\leq C_1\exp\left(-\frac{C_2\lambda^2}{\mathcal{C}_2^2\vert x-y\vert^{2\xi}}\right),
$$
which completes the proof of \eqref{spacep}. 

Define 
$$R_{q_1}=\int_0^{q_1}\int_{-1}^1 [G_{t-r}(x-y)-G_{s-r}(x-y)]\sigma(r,y,u(r,y))F(dydr),$$
where $0\leq q_1\leq s$. Note $R_{q_1}$ is a continuous $\mathcal{F}_{q_1}^F$-adapted martingale with $R_0=0$. Also define
$$S_{q_2}=\int_s^{s+{q_2}}\int_{-1}^1 G_{t-r}(x-y)\sigma(r,y,u(r,y))F(dydr),$$
where $0\leq q_2\leq t-s$. Note $S_{q_2}$ is a continuous $\mathcal{F}_{q_2}^F$-adapted martingale with $S_0=0$. Thus, both $R_{q_1}$ and $S_{q_2}$ are time changed Brownian motions, i.e.,
$$R_s=B_{\langle R\rangle_s}\quad \text{and}\quad S_{t-s}=B'_{\langle S\rangle_{t-s}},$$
where $B$, $B'$ are two different Brownian motions. Note that $N(t,x)-N(s,x)=R_s+S_{t-s}$, then
$$P(N(t,x)-N(s,x)>\lambda)\leq P(R_s>\lambda/2)+P(S_{t-s}>\lambda/2).$$
\eqref{I2equ} and \eqref{I1equ} give an upper bound on the time changes as
$$\langle R\rangle_s\leq C\mathcal{C}_2^2(t-s)^{2\zeta}\quad \text{and} \quad \langle S\rangle_{t-s}\leq C\mathcal{C}_2^2(t-s)^{2\zeta}.$$
By the reflection principle for the Brownian motions $B_{\langle R\rangle_s}$ and $B'_{\langle S\rangle_{t-s}}$,
\begin{align*}
P(N(t,x)-N(s,x)>\lambda)&\leq P\left(B_{\langle R\rangle_s}>\lambda/2\right)+P\left(B'_{\langle S\rangle_{t-s}}>\lambda/2\right)\\
&\leq 2P\left(\sup_{r\leq C\mathcal{C}_2^2\vert t-s\vert^{2\zeta}}B_r>\frac{\lambda}{2}\right)=4P\left(B_{C\mathcal{C}_2^2\vert t-s\vert^{2\zeta}}>\frac{\lambda}{2}\right)\\
&\leq C_3\exp\left(-\frac{C_4\lambda^2}{\mathcal{C}_2^2\vert t-s\vert^{2\zeta}}\right).
\end{align*}
In addition,
\begin{align*}
P(-N(t,x)+N(s,x)>\lambda)&\leq P(R_s<-\lambda/2)+P(S_{t-s}<-\lambda/2)\\
&\leq 4P\left(B_{C\mathcal{C}_2^2\vert t-s\vert^{2\zeta}}<-\frac{\lambda}{2}\right)\\
&\leq C_3\exp\left(-\frac{C_4\lambda^2}{\mathcal{C}_2^2\vert t-s\vert^{2\zeta}}\right).
\end{align*}
Consequently, for $\forall\zeta\in(0,\alpha)$, 
$$
P(\vert N(t,x)-N(s,x)\vert>\lambda)\leq C_3\exp\left(-\frac{C_4\lambda^2}{\mathcal{C}_2^2\vert t-s\vert^{2\zeta}}\right),
$$
which completes the proof of \eqref{timep}.
\end{proof}

\begin{definition}
Given a grid
$$\mathbb{G}_n=\left\lbrace\left(\frac{j}{2^{2n}},\frac{k}{2^{n}}\right): 0\leq j\leq 2^{2n},0\leq k\leq2^n,j,k\in\Z\right\rbrace,$$
and we write
$$\left(t_j^{(n)},x_k^{(n)}\right)=\left(\frac{j}{2^{2n}},\frac{k}{2^n}\right).$$
Two points $\left(t_{j_i}^{(n)},x_{k_i}^{(n)}\right):i=1,2$ are called \textbf{nearest neighbors} if either
$$j_1=j_2 \quad\text{and}\quad \vert k_1-k_2\vert=1\quad\textbf{or}\quad \vert j_1-j_2\vert=1\quad\text{and}\quad k_1=k_2.$$
\end{definition}

The following lemma is analogous to Lemma 3.4 in \cite{athreya2021small} and Lemma 2.5 in \cite{foondun2022small}, which is crucial for estimating small ball probabilities.

\begin{lemma}\label{larged}There exist constants $C_5,C_6>0$ depending on $\gamma$ such that for all $\beta,\lambda,\varepsilon>0$ and $\beta\varepsilon^4\leq 1$, we have
\begin{equation*}
P\left(\sup_{\substack{0\leq t\leq\beta\varepsilon^4\\ x\in\left[0,\varepsilon^2\right]}}\vert N(t,x)\vert>\lambda\varepsilon^{2-\gamma}\right)\leq \frac{C_5}{1\wedge \sqrt{\beta}}\exp\left(-\frac{C_6\lambda^2}{ \mathcal{C}_2^2\beta^\frac{2-\gamma}{2}}\right).
\end{equation*}
\end{lemma}
\begin{proof}
Let us first consider the case $\beta\geq 1$. For $n\geq 0$, consider the grid
$$\mathbb{G}_n=\left\lbrace\left(\frac{j}{2^{2n}},\frac{k}{2^{n}}\right): 0\leq j\leq \beta \varepsilon^42^{2n},0\leq k\leq\varepsilon^22^n,j,k\in\Z\right\rbrace.$$
Let
\begin{equation}
\label{n0}
n_0=\left\lceil \log_2\left(\beta^{-1/2}\varepsilon^{-2}\right)\right\rceil,
\end{equation}
and for $n< n_0$, $\mathbb{G}_n$ contains only the origin. For $n\geq n_0$, the grid $\mathbb{G}_n$ has at most $\left(\beta\varepsilon^4 2^{2n}+1\right)\cdot\left(\varepsilon^2 2^n+1\right)\leq 2^5\cdot 2^{3(n-n_0)}$ many points. Because $4\zeta \wedge 2\xi<4\alpha<2-\gamma$, we can choose two parameters $0<\delta_1(\gamma)<\delta_2(\gamma)<\frac{2-\gamma}{2}$, and define
\begin{equation*}
\delta := \delta_2-\delta_1,
\end{equation*}
which satisfies the following constraint
\begin{equation}
\label{deltaconstraint}
4\zeta\wedge 2\xi=2-\gamma-2\delta.
\end{equation}

We now define another constant
\begin{equation}
\label{M}
\mathcal{M}=\frac{1-2^{-\delta_1}}{4\cdot 2^{\delta n_0}},
\end{equation}
and consider the event
$$A(n,\lambda)=\left\lbrace\vert N(p)-N(q)\vert\leq \lambda \mathcal{M}\varepsilon^{2-\gamma} 2^{-\delta_1 n}2^{\delta_2 n_0}\text{ for all nearest neighbors $p,q\in \mathbb{G}_n$}\right\rbrace.$$
If $p,q\in \mathbb{G}_n$ have different spatial variables, \eqref{spacep} implies
\[P\left(\vert N(p)-N(q)\vert> \lambda \mathcal{M}\varepsilon^{2-\gamma} 2^{-\delta_1n}2^{\delta_2n_0}\right)\leq C_1\exp\left(-\frac{C_2\lambda^2\mathcal{M}^2\varepsilon^{4-2\gamma}}{2^{-2n\xi}\mathcal{C}_2^2}2^{-2\delta_1n}2^{2\delta_2n_0}\right).\]
If $p,q\in \mathbb{G}_n$ have different time variables, \eqref{timep} implies
\[P\left(\vert N(p)-N(q)\vert> \lambda \mathcal{M}\varepsilon^{2-\gamma} 2^{-\delta_1n}2^{\delta_2n_0}\right)\leq C_3\exp\left(-\frac{C_4\lambda^2\mathcal{M}^2\varepsilon^{4-2\gamma}}{2^{-4n\zeta}\mathcal{C}_2^2}2^{-2\delta_1n}2^{2\delta_2n_0}\right).\]

Therefore, a union bound gives
\begin{align*}
&P(A^c(n,\lambda))\leq \sum_{\substack{p,q\in \mathbb{G}_n\\ \text{nearest neighbors}}}P\left(\vert N(p)-N(q)\vert> \lambda \mathcal{M}\varepsilon^{2-\gamma} 2^{-\delta_1n}2^{\delta_2n_0}\right)\\
&\leq C2^{3(n-n_0)}\exp\left(-\frac{C'\lambda^2\mathcal{M}^2\varepsilon^{4-2\gamma}}{\mathcal{C}_2^2}2^{n(4\zeta\wedge 2\xi)}2^{-2\delta_1n}2^{2\delta_2n_0}\right)\\
&=C2^{3(n-n_0)}\exp\left(-\frac{C'\lambda^2\mathcal{M}^2}{\mathcal{C}_2^2}\left(\varepsilon^{4-2\gamma} 2^{n_0(2-\gamma)}\right)2^{n(4\zeta\wedge 2\xi)}2^{-2\delta_1n}2^{-n_0(2-\gamma)}2^{2\delta_2n_0}\right)\\
&\leq C2^{3(n-n_0)}\exp\left(-\frac{C'\lambda^2\mathcal{M}^2}{\mathcal{C}_2^2\beta^{\frac{2-\gamma}{2}}}2^{(4\zeta\wedge 2\xi-2\delta_1)(n-n_0)}\right),
\end{align*}
where $C,C'$ are positive constants depending only on $\gamma$. The last inequality follows from that $\varepsilon^{4-2\gamma} 2^{n_0(2-\gamma)}\geq\beta^{-\frac{(2-\gamma)}{2}}$ by the definition of $n_0$ in \eqref{n0}, and our choices of $\delta_1,\delta_2$ and $\delta$ in \eqref{deltaconstraint}.

Let $A(\lambda)=\bigcap_{n\geq n_0}A(n,\lambda)$ and we can bound $P(A^c(\lambda))$ by summing $P(A^c(n,\lambda))$ for all $n\geq n_0$,
\begin{align*}
P\left(A^c(\lambda)\right)&\leq\sum_{n\geq n_0}P\left(A^c(n,\lambda)\right)\\
&\leq \sum_{n\geq n_0}C2^{3(n-n_0)}\exp\left(-\frac{C'\lambda^2\mathcal{M}^2}{\mathcal{C}_2^2\beta^{\frac{2-\gamma}{2}}}2^{(4\zeta\wedge 2\xi-2\delta_1)(n-n_0)}\right)\\
&\leq C_5\exp\left(-\frac{C'\lambda^2\mathcal{M}^2}{\mathcal{C}_2^2\beta^{\frac{2-\gamma}{2}}}\right).
\end{align*}
From definitions \eqref{n0} and \eqref{M}, we have
\begin{align*}
P\left(A^c(\lambda)\right)&\leq C_5\exp\left(-\frac{C'\lambda^22^{-2\delta n_0}}{\mathcal{C}_2^2\beta^{\frac{2-\gamma}{2}}}\right)\\
&\leq C_5\exp\left(-\frac{C_6\lambda^2\left(\beta\varepsilon^4\right)^\delta}{\mathcal{C}_2^2\beta^{\frac{2-\gamma}{2}}}\right)
\end{align*}
for any $\delta\in\left(0,\frac{2-\gamma}{2}\right)$, we thus obtain
\begin{align*}
P\left(A^c(\lambda)\right)&\leq \inf_{\delta\in\left(0,\frac{2-\gamma}{2}\right)}C_5\exp\left(-\frac{C_6\lambda^2\left(\beta\varepsilon^4\right)^\delta}{\mathcal{C}_2^2\beta^{\frac{2-\gamma}{2}}}\right)\\
&\leq C_5\exp\left(-\frac{C_6\lambda^2}{\mathcal{C}_2^2\beta^{\frac{2-\gamma}{2}}}\right),
\end{align*}
where both $C_5$ and $C_6$ depend on $\gamma.$

Now we consider a point $(t,x)$, which is in a grid $\mathbb{G}_n$ for some $n \geq n_0$. From arguments similar to page 128 of \cite{dalang2009minicourse}, we can find a sequence of points from the origin to $(t,x)$ as $(0,0)= p_0,p_1,...,p_k = (t,x)$ such that each pair is the nearest neighbor in some grid $G_m, n_0\leq m \leq n$, and at most 4 such pairs are nearest neighbors in any given grid. On the event $A(\lambda)$, we have
\begin{align*}
\vert N(t,x)\vert&\leq \sum_{j=0}^{k-1}\vert N(p_j)-N(p_{j+1})\vert\leq 4\sum_{n\geq n_0}\lambda \mathcal{M}\varepsilon^{2-\gamma} 2^{-\delta_1n}2^{\delta_2n_0}\\
&= 4\lambda\varepsilon^{2-\gamma}2^{\delta_2n_0}\frac{1-2^{-\delta_1}}{4\cdot 2^{\delta n_0}}\sum_{n\geq n_0}2^{-\delta_1n}\\
&\leq\lambda\varepsilon^{2-\gamma}.
\end{align*}

Points in $\bigcup\limits_n\mathbb{G}_n$ are dense in $\left[0,\beta\varepsilon^4\right]\times\left[0,\varepsilon^2\right]$, and we may extend $N(t,x)$ to a continuous version. Therefore, for $\beta\geq 1$,
$$
P\left(\sup_{\substack{0\leq t\leq\beta\varepsilon^4\\ x\in\left[0,\varepsilon^2\right]}}\vert N(t,x)\vert>\lambda\varepsilon^{2-\gamma}\right)\leq C_5\exp\left(-\frac{C_6\lambda^2}{\mathcal{C}_2^2\beta^{\frac{2-\gamma}{2}}}\right).
$$
For $0<\beta<1$, we divide the interval $[0,\varepsilon^2]$ into $\frac{1}{\sqrt{\beta}}$ pieces each of length $\sqrt{\beta}\varepsilon^2$, and then a union bound implies that
\[P\left(\sup_{\substack{0\leq t\leq\beta\varepsilon^4\\ x\in\left[0,\varepsilon^2\right]}}\vert N(t,x)\vert>\lambda\varepsilon^{2-\gamma}\right)\leq \frac{1}{\sqrt{\beta}} P\left(\sup_{\substack{0\leq t\leq\beta\varepsilon^4\\ x\in\left[0,\sqrt{\beta}\varepsilon^2\right]}}\vert N(t,x)\vert>\lambda\varepsilon^{2-\gamma}\right)\]
\[=\frac{1}{\sqrt{\beta}}P\left(\sup_{\substack{0\leq t\leq(\sqrt{\beta}\varepsilon^2)^2\\ x\in\left[0,\sqrt{\beta}\varepsilon^2\right]}}\vert N(t,x)\vert>\frac{\lambda}{\beta^{\frac{2-\gamma}{4}}}\left(\beta^{1/4}\varepsilon\right)^{2-\gamma}\right)\leq \frac{C_5}{\sqrt{\beta}}\exp\left(-\frac{C_6\lambda^2}{ \mathcal{C}_2^2\beta^{\frac{2-\gamma}{2}}}\right).\]
As a result,
\[
P\left(\sup_{\substack{0\leq t\leq\beta\varepsilon^4\\ x\in \left[0,\varepsilon^2\right]}}\vert N(t,x)\vert>\lambda\varepsilon^{2-\gamma}\right)\leq \frac{C_5}{1\wedge \sqrt{\beta}}\exp\left(-\frac{C_6\lambda^2}{\mathcal{C}_2^2\beta^{\frac{2-\gamma}{2}}}\right).
\]
\end{proof}

\begin{remark}
\label{largeremark}
If we assume that $\sigma$ in \eqref{noiseterm} satisfies $\vert \sigma(s,y,u)\vert\leq C\left(\beta\varepsilon^4\right)^{\frac{2-\gamma}{4}}$, then the probability in Lemma \ref{larged} is bounded above by
\[\frac{C_5}{1\wedge \sqrt{\beta}}\exp\left(-\frac{C_6\lambda^2}{(\beta\varepsilon^2)^{2-\gamma}}\right),\]
which can be proved similarly to the above lemma.
\end{remark}

\section{Proof of Proposition \ref{prop}}\label{proofprop}
The following lemma gives a lower bound for variance of the noise term $N(t_1,x_k)$ and an upper bound on the decay of covariance between two random variables $N(t_1,x_k), N(t_1,x_{k'})$ as $\vert k-k'\vert$ increases.
\begin{lemma}
\label{varbound}
Consider noise terms $N(t_1,x_k),N(t_1,x_{k'})$ and a deterministic $\sigma(t,x,u)=\sigma(t,x)$, then there exist positive constants $C_7,C_8$ and $C_9$ depending only on $\mathcal{C}_1,\mathcal{C}_2$ and $\gamma$, such that
\begin{equation*}
C_7t_1^{\frac{2-\gamma}{2}}\leq {\rm Var}(N(t_1,x_k))\leq C_8t_1^{\frac{2-\gamma}{2}}
\end{equation*}
and
\begin{equation*}
{\rm Cov}(N(t_1,x_k),N(t_1,x_{k'}))\leq C_9t_1\left\vert (k-k')\varepsilon^2\right\vert^{-\gamma}.
\end{equation*}
\end{lemma}
\begin{proof} We use the assumption of $\sigma$ in \eqref{assume2} to achieve
\begin{align*}
&\textrm{Var}(N(t_1,x_k))\\
&=\int_0^{t_1}\int_{-1}^1\int_{-1}^1G_{t_1-s}(x_k-y)G_{t_1-s}(x_k-z)\sigma(s,y)\sigma(s,z)\Lambda(y-z)dydzds\\
&\geq \mathcal{C}_1^2\int_0^{t_1}\sum_{n=0}^\infty q_ne^{-2\pi^2n^2(t_1-s)}ds\\
&\geq C\mathcal{C}_1^2\left(\int_0^{t_1}\int_0^\infty x^{\gamma-1}e^{-2\pi^2x^2(t_1-s)}dxds\right).
\end{align*}

The last inequality follows from that $q_n$ in \eqref{qn} and the fact that $n^{\gamma-1}e^{-2\pi^2n^2(t_1-s)}$ decreases as $n$ increases. Using Fubini's theorem and changing variable with $w=2\pi^2 x^2 t_1$, we continue as follows,
$$\textrm{Var}(N(t_1,x_k))\geq C\mathcal{C}_1^2\int_0^\infty\frac{1-e^{-2\pi^2x^2t_1}}{2\pi^2x^{3-\gamma}}dx=C\mathcal{C}_1^2t_1^{\frac{2-\gamma}{2}}\int_0^\infty \frac{1-e^{-w}}{w^{\frac{4-\gamma}{2}}}dw\geq C_7t_1^{\frac{2-\gamma}{2}}.$$
The last integral converges with $\gamma\in(0,1)$, which completes the proof of the lower bound. 

On the other hand,
\begin{align*}
&\textrm{Var}(N(t_1,x_k))\\
&=\int_0^{t_1}\int_{-1}^1\int_{-1}^1G_{t_1-s}(x_k-y)G_{t_1-s}(x_k-z)\sigma(s,y)\sigma(s,z)\Lambda(y-z)dydzds\\
&\leq \mathcal{C}_2^2\int_0^{t_1}\sum_{n=0}^\infty q_ne^{-2\pi^2n^2(t_1-s)}ds\\
&\leq C'\mathcal{C}_2^2\left(q_0t_1^\frac{2-\gamma}{2}+\int_0^{t_1}\int_0^\infty x^{\gamma-1}e^{-2\pi^2x^2(t_1-s)}dxds\right).
\end{align*}
The last inequality follows from that $q_n$ in \eqref{qn} and the fact that $t_1<t_1^{\frac{2-\gamma}{2}}$ for $t_1<1$. Using the same argument as above, we end up with 
$$\textrm{Var}(N(t_1,x_k))\leq C_8t_1^{\frac{2-\gamma}{2}}.$$
In addition, we use the fact $1-e^{-x}\leq x$ to derive the upper bound of covariance between $N(t_1,x_k)$ and $N(t_1,x_{k'})$ when $k\neq k'$,
\begin{align*}
&\textrm{Cov}(N(t_1,x_k),N(t_1,x_{k'}))\\
&\leq \mathcal{C}_2^2\left(\sum_{n=0}^\infty \E\left[\left(\int_0^{t_1}e^{-\pi^2n^2(t_1-s)}\beta_n(ds)\right)^2\right] q_n\cos\left(n\pi (k-k')\varepsilon^2\right)\right)\\
&\leq \mathcal{C}_2^2\left(t_1q_0+\sum_{n=1}^\infty \frac{1-e^{-2\pi^2n^2t_1}}{2\pi^2n^2}q_n\cos\left(n\pi (k-k')\varepsilon^2\right)\right)\\
&\leq Ct_1\sum_{n=0}^\infty q_n\cos(n\pi (k-k')\varepsilon^2)\leq C_9t_1\vert (k-k')\varepsilon^2\vert^{-\gamma},
\end{align*}
\end{proof}
where the last equality follows by the definition of $\Lambda(x)$ from \eqref{lambdafourier} and \eqref{Lambdabound}.

\textbf{Proof of Proposition \ref{prop}(a)} Recall $F_n$ in \eqref{Fn}. The Markov property of $u(t,\cdot)$ (see page 247 in \cite{da2014stochastic}) implies
\begin{equation*}
P\left(F_{j}\vert\sigma\{u(t_i,\cdot)\}_{0\leq i<j}\right)=P\left(F_{j}\vert u(t_{j-1},\cdot)\right).
\end{equation*}
If we can show that $P\left(F_{j}\vert u(t_{j-1},\cdot)\right)$ has a uniform bound $\textbf{C}_4\exp\left(-\frac{\textbf{C}_5}{t_1^{\gamma/2}}\right)$ that does not depend on $j$, then the same bound holds for the conditional probability $P\left(F_{j}\Big\vert \bigcap\limits_{k=0}^{j-1}F_{k}\right)$, which is conditioned on a realization of $u(t_k,\cdot),0\leq k<j$. Thus, it is enough to prove
\begin{equation}
\label{F1}
P\left(F_{1}\right)\leq\textbf{C}_4\exp\left(-\frac{\textbf{C}_5}{t_1^{\gamma/2}}\right),
\end{equation}
where $\textbf{C}_4$ and $\textbf{C}_5$ do not depend on $u_0$ and $\vert u_0(x)\vert<t_1^{(2-\gamma)/4}$ for every grid point $x$ in (3.3). 

Now we show how to find a uniform bound for $P(F_1)$ with starting from considering the truncated function
$$f_\varepsilon(x)=\begin{cases}
x & \vert x\vert\leq t_1^{\frac{2-\gamma}{4}}\\
\frac{x}{\vert x\vert}\cdot t_1^{\frac{2-\gamma}{4}} & \vert x\vert> t_1^{\frac{2-\gamma}{4}}
\end{cases},$$
and, particularly, we have $\vert f_\varepsilon(x)\vert\leq t_1^{\frac{2-\gamma}{4}}$. Consider the following two equations
\begin{equation*}
\label{shef}
\partial_tv(t,x)=\partial_x^2v(t,x)+\sigma(t,x,f_\varepsilon(v(t,x))) \dot{F}(t,x)
\end{equation*}
and
\begin{equation}
\label{shefreeze}
\partial_tv_g(t,x)=\partial_x^2v_g(t,x)+\sigma(t,x,f_\varepsilon(u_0(x))) \dot{F}(t,x)
\end{equation}
with the same initial $u_0(x)$. We can decompose $v(t,x)$ by
\[v(t,x)=v_g(t,x)+D(t,x)\]
with
\[D(t,x)=\int_0^t\int_{-1}^1 G_{t-s}(x-y)[\sigma(s,y,f_\varepsilon(v(s,y)))-\sigma(s,y,f_\varepsilon(u_0(y)))]F(dyds).\]

We define a new sequence of events for $n=1$,
\[H_j=\left\lbrace\vert v(t_1,x_j)\vert\leq t_1^{\frac{2-\gamma}{4}}\right\rbrace.\]
Clearly, the property of $f_\varepsilon(x)$ implies
$$F_{1}=\bigcap_{j=-n_1+1}^{n_1-1}H_j.$$
Moreover, we define another two sequences of events
$$A_j=\left\lbrace\vert v_g(t_1,x_j)\vert\leq 2t_1^{\frac{2-\gamma}{4}}\right\rbrace\quad\text{and}\quad B_j=\left\lbrace\vert D(t_1,x_j)\vert>t_1^{\frac{2-\gamma}{4}}\right\rbrace.$$
If $\vert D(t_1,x_j)\vert\leq t_1^{\frac{2-\gamma}{4}}$ and $\vert v_g(t_1,x_j)\vert> 2t_1^{\frac{2-\gamma}{4}}$, then $\vert v(t_1,x_j)\vert>t_1^{\frac{2-\gamma}{4}}$. As a result,
$$H_j^c\supset A_j^c\cap B_j^c,$$
which leads to
\begin{equation}
\label{sumprob}
\begin{split}
P(F_{1})&= P\left(\bigcap_{j=-n_1+1}^{n_1-1}H_j\right)\leq P\left(\bigcap_{j=-n_1+1}^{n_1-1}[A_j\cup B_j]\right)\\
&\leq P\left(\left(\bigcap_{j=-n_1+1}^{n_1-1}A_j\right)\bigcup\left(\bigcup_{j=-n_1+1}^{n_1-1}B_j\right)\right)\\
&\leq P\left(\bigcap_{j=-n_1+1}^{n_1-1}A_j\right)+P\left(\bigcup_{j=-n_1+1}^{n_1-1}B_j\right)\\
&\leq P\left(\bigcap_{j=-n_1+1}^{n_1-1}A_j\right)+\sum_{j=-n_1+1}^{n_1-1} P(B_j),
\end{split}
\end{equation}
where the second inequality can be shown by using induction. 

Recall the Lipschitz property on the third variable of $\sigma(t,x,u)$ and $\vert f_\varepsilon(v(t,x))-f_\varepsilon(u_0(x))\vert\leq 2t_1^{\frac{2-\gamma}{4}}$. Applying Remark \ref{largeremark} on $D(t,x)$ will give the inequality
\begin{equation*}
P(B_j)\leq \frac{C_5}{1\wedge \sqrt{c_0}}\exp\left(-\frac{1}{4}\mathcal{D}^{-2}C_6t_1^{-\frac{2-\gamma}{2}}\right),
\end{equation*}
so that a bound \eqref{jbound} on $j$ yields a union bound of probabilities
\begin{equation}
\label{probB}
P\left(\bigcup_{j=-n_1+1}^{n_1-1}B_j\right) \leq \sum_{j=-n_1+1}^{n_1-1}P(B_j)\leq\frac{C_5}{(1\wedge \sqrt{c_0})\varepsilon^2}\exp\left(-\frac{C_6}{ 4\mathcal{D}^2t_1^{1-\gamma/2}}\right).
\end{equation}

We define a sequence of events involving $v_g$ to compute the upper bound for $P\left(\bigcap\limits_{j=-n_1+1}^{n_1-1}A_j\right)$, 
$$I_j=\left\lbrace\vert v_g(t_1,x_k)\vert\leq 2t_1^{\frac{2-\gamma}{4}}\text{~for all $-n_1+1\leq k\leq j$}\right\rbrace,\quad I_{-n_1}=\Omega.$$
Then we can write $P\left(\bigcap\limits_{j=-n_1+1}^{n_1-1}A_j\right)$ in terms of conditional probabilities,
\begin{equation}
\label{Acond}
P\left(\bigcap_{j=-n_1+1}^{n_1-1}A_j\right)=P(I_{n_1-1})=P(I_{-n_1})\prod_{j=-n_1+1}^{n_1-1}\frac{P(I_j)}{P(I_{j-1})}=\prod_{j=-n_1+1}^{n_1-1}P(I_j\vert I_{j-1}).
\end{equation}
Let $\mathcal{G}_j$ be the $\sigma-$algebra generated by 
$$N_\varepsilon(t_1,x_k):=\int_0^{t_1}\int_{-1}^1 G(t_1-s,x_k-y)\sigma(s,y,f_\varepsilon(u_0(y)))F(dyds),\quad -n_1+1\leq k\leq j,$$
where $N_\varepsilon(t_1,x_k)$ is the noise term of $v_g(t_1,x_k)$ in \eqref{shefreeze}. If we can show that there is a uniform bound for $P(I_j\vert \mathcal{G}_{j-1})$, then the bound holds for the conditional probability $P(I_j\vert I_{j-1})$, which is conditioned on a realization of $N_\varepsilon(t_1,x_k),k=-n_1+1,...,j-1$. Notice that $\sigma(s,y,f_\varepsilon(u_0(y)))$ is deterministic and $\sigma$ is uniformly bounded. By Lemma \ref{varbound}, we have
\begin{equation*}
\textrm{Var}(N_\varepsilon(t_1,x_k))\geq C_7 t_1^{\frac{2-\gamma}{2}},
\end{equation*}
and one can decompose
$$v_g(t_1,x_j)=\int_{-1}^{1}G_{t_1}(x-y)u_0(y)dy+X+Y,$$
where $X=\E[N_\varepsilon(t_1,x_j)\vert \mathcal{G}_{j-1}]$ is a Gaussian random variable, which can be written as
$$X=\sum_{k=-n_1+1}^{j-1}\eta_k^{(j)}N_\varepsilon(t_1,x_k)$$
for some coefficients $\left(\eta_k^{(j)}\right)_{k=-n_1+1}^{j-1}$. Then the conditional variance equals
\begin{align*}
&\textrm{Var}(Y\vert\mathcal{G}_{j-1})=\E [(N_\varepsilon(t_1,x_j)-X)^2\vert\mathcal{G}_{j-1}]-(\E[N_\varepsilon(t_1,x_j)-X\vert\mathcal{G}_{j-1}])^2\\
&=\E[(N_\varepsilon(t_1,x_j)-\E[N_\varepsilon(t_1,x_j)\vert \mathcal{G}_{j-1}])^2\vert\mathcal{G}_{j-1}]=\textrm{Var}(N_\varepsilon(t_1,x_j)\vert \mathcal{G}_{j-1}).
\end{align*}
Since $Y=N_\varepsilon(t_1,x_j)-X$ is independent of $\mathcal{G}_{j-1}$, we write $\textrm{Var}(Y)$ as
$$\textrm{Var}(Y)=\textrm{Var}(Y\vert\mathcal{G}_{j-1})=\textrm{Var}(N_\varepsilon(t_1,x_j)\vert \mathcal{G}_{j-1}).$$
In fact, for a Gaussian random variable $Z \sim N(\mu,\sigma^2)$ and any $a > 0$, the probability $P(\vert Z\vert \leq a)$ is maximized when $\mu = 0$, thus
\begin{equation*}
P(I_j\vert \mathcal{G}_{j-1})\leq P\left(\vert v_g(t_1,x_j)\vert\leq 2t_1^{\frac{2-\gamma}{4}}\bigg| \mathcal{G}_{j-1}\right)\leq P\left(\vert Z'\vert\leq \frac{2t_1^{\frac{2-\gamma}{4}}}{\sqrt{\textrm{Var}(N_\varepsilon(t_1,x_j)\vert \mathcal{G}_{j-1})}}\right)
\end{equation*}
where $Z'\sim N(0,1)$. Let's use the notation $\textrm{SD}$ to denote the standard deviation of a random variable. By the Minkowski inequality,
$$\textrm{SD}(X)\leq \sum_{k=-n_1+1}^{j-1}\left\vert\eta_k^{(j)}\right\vert\cdot \textrm{SD}(N_\varepsilon(t_1,x_k)),$$
and
$$\textrm{SD}(N_\varepsilon(t_1,x_j))\leq \textrm{SD}(X)+\textrm{SD}(Y).$$
If we can control coefficients by restricting
$$\sum_{k=-n_1+1}^{j-1}\left\vert\eta_k^{(j)}\right\vert<\frac{C_7}{C_8},$$
then by Lemma \ref{varbound}, we obtain
\begin{equation*}
\begin{split}
\textrm{SD}(X)&\leq\sum_{k=-n_1+1}^{j-1}\left\vert\eta_k^{(j)}\right\vert\cdot \textrm{SD}(N_\varepsilon(t_1,x_k))\leq\left(\sum_{k=-n_1+1}^{j-1}\left\vert\eta_k^{(j)}\right\vert\right)\cdot\sup_{-n_1+1\leq k<j}\textrm{SD}(N_\varepsilon(t_1,x_k))\\
&<\frac{C_7}{C_8}\cdot C_8t_1^{\frac{2-\gamma}{4}}=C_7t_1^{\frac{2-\gamma}{4}}.
\end{split}
\end{equation*}
Therefore, $\textrm{SD}(Y)$ can be bounded below by
\begin{equation*}
\textrm{SD}(Y)\geq \textrm{SD}(N_\varepsilon(t_1,x_j))-\textrm{SD}(X)> Ct_1^{\frac{2-\gamma}{4}},
\end{equation*}
so that we can derive the uniform upper bound of $P(I_j\vert \mathcal{G}_{j-1})$,
\begin{equation*}
\label{probA}
\begin{split}
P(I_j\vert \mathcal{G}_{j-1})&\leq P\left(\vert Z'\vert\leq \frac{2t_1^{\frac{2-\gamma}{4}}}{\sqrt{\textrm{Var}(N_\varepsilon(t_1,x_j)\vert \mathcal{G}_{j-1})}}\right)\\
&\leq P\left(\vert Z'\vert\leq \frac{2t_1^{\frac{2-\gamma}{4}}}{\sqrt{Ct_1^{\frac{2-\gamma}{2}}}}\right)\\
&=P(\vert Z'\vert\leq C').
\end{split}
\end{equation*}
A bound \eqref{jbound} on $j$ and \eqref{Acond} together yield
\begin{equation}
\label{probAbound}
P\left(\bigcap_{j=-n_1+1}^{n_1-1}A_j\right)\leq C^{\varepsilon^{-2}}=C\exp\left(-\frac{C'}{\varepsilon^2}\right).
\end{equation}
The following lemma shows how to select $c_0$ to make $\sum\limits_{k=-n_1+1}^{j-1}\left\vert\eta_k^{(j)}\right\vert< \frac{C_7}{C_8}$, which completes the proof.

\begin{lemma}\label{coeffbound}
For a given $\varepsilon>0$, we may choose $c_0>0$ in \eqref{c0} such that
\begin{equation*}
\sum_{k=-n_1+1}^{j-1}\left\vert\eta_k^{(j)}\right\vert< \frac{C_7}{C_8}.
\end{equation*}
\end{lemma}
\begin{proof}
Since $Y$ and $\mathcal{G}_{j-1}$ are independent, 
$$\textrm{Cov}(Y,N_\varepsilon(t_1,x_k))=0\text{~for $k=-n_1+1,...,j-1,$}$$
and for $l=-n_1+1,...,j-1$, we have
\begin{equation*}
\textrm{Cov}(N_\varepsilon(t_1,x_j),N_\varepsilon(t_1,x_l))=\textrm{Cov}(X,N_\varepsilon(t_1,x_l))=\sum_{k=-n_1+1}^{j-1}\eta^{(j)}_k \textrm{Cov}(N_\varepsilon(t_1,x_k),N_\varepsilon(t_1,x_l)).
\end{equation*}
Consider the vector ${\bf \eta}=\left(\eta^{(j)}_{-n_1+1},...,\eta^{(j)}_{j-1}\right)^T$ and vector
$${\bf x}=(\textrm{Cov}(N_\varepsilon(t_1,x_j),N_\varepsilon(t_1,x_{-n_1+1})),...,\textrm{Cov}(N_\varepsilon(t_1,x_j),N_\varepsilon(t_1,x_{j-1})))^T,$$
and let ${\bf \Sigma}=[\textrm{Cov}\left(N_\varepsilon(t_1,x_k),N_\varepsilon(t_1,x_l)\right)]_{-n_1+1\leq k,l\leq j-1}$ be the covariance matrix of $N_\varepsilon(t_1,\cdot)$, then we get
$${\bf \eta}={\bf \Sigma}^{-1}{\bf x}.$$
Let $\vert\vert\cdot\vert\vert_{1,1}$ be the operator norm on matrices, then
$$\vert\vert{\bf \eta}\vert\vert_{l^1}=\vert\vert{\bf \Sigma}^{-1}{\bf x}\vert\vert_{l^1}\leq\vert\vert{\bf \Sigma}^{-1}\vert\vert_{1,1}\vert\vert{\bf x}\vert\vert_{l^1}.$$
We rewrite ${\bf \Sigma}={\bf D}{\bf T}{\bf D}$, where ${\bf D}$ is a diagonal matrix with diagonal entries $\sqrt{\textrm{Var}(N_\varepsilon(t_1,x_k))}$, and ${\bf T}$ is the correlation matrix with entries
$$t_{kl}=\frac{\textrm{Cov}(N_\varepsilon(t_1,x_k),N_\varepsilon(t_1,x_l))}{\sqrt{\textrm{Var}(N_\varepsilon(t_1,x_k))}\sqrt{\textrm{Var}(N_\varepsilon(t_1,x_l))}}.$$
Thanks to Lemma \ref{varbound}, for $k\neq l$, $t_{kl}$ can be bounded above by
$$ \vert t_{kl}\vert\leq \frac{C_9c_0\varepsilon^4\left\vert (k-l)\varepsilon^2\right\vert^{-\gamma}}{C_7(c_0\varepsilon^4)^{\frac{2-\gamma}{2}}}=\frac{C_9}{C_7}\cdot\left(\frac{\sqrt{c_0}}{\vert k-l\vert}\right)^\gamma.$$
Define ${\bf A}={\bf I-T}$. Because ${\bf A}$ has zero diagonal entries, we can bound $\vert\vert{\bf A}\vert\vert_{1,1}$ by
\begin{align*}
\vert\vert{\bf A}\vert\vert_{1,1}&\leq 2\sum_{k-l=1}^{\left[\varepsilon^{-2}\right]}\vert t_{kl}\vert\leq \frac{2C_9\left(\sqrt{c_0}\right)^\gamma}{C_7}\int_0^{\varepsilon^{-2}}x^{-\gamma}dx\\
&=\frac{2C_9\left(\sqrt{c_0}\right)^\gamma}{C_7}\cdot \frac{\left(\varepsilon^2\right)^{\gamma-1}}{1-\gamma}=\frac{2C_9}{C_7(1-\gamma)}\cdot\frac{\left(c_0\varepsilon^4\right)^{\gamma/2}}{\varepsilon^2}.
\end{align*}
For any $\varepsilon>0$, $c_0$ is selected to satisfy
$$\vert\vert{\bf A}\vert\vert_{1,1}\leq\frac{2C_9}{C_7(1-\gamma)}\cdot\frac{(c_0\varepsilon^4)^{\gamma/2}}{\varepsilon^2}<\frac{2C_7}{2C_7+C_8},$$
which becomes
$$c_0<\left(\frac{C_7^2(1-\gamma)}{2C_7C_9+C_8C_9}\right)^{\frac{2}{\gamma}}\varepsilon^{4/\gamma-4}.$$
Let's denote $\mathcal{C}=\left(\frac{C_7^2(1-\gamma)}{2C_7C_9+C_8C_9}\right)^{\frac{2}{\gamma}}$ in \eqref{c0}. Choosing such $c_0$ will make
$$\vert\vert{\bf T}^{-1}\vert\vert_{1,1}=\vert\vert{\bf (I-A)}^{-1}\vert\vert_{1,1}\leq \frac{1}{1-\vert\vert{\bf A}\vert\vert_{1,1}}\leq \frac{2C_7+C_8}{C_8},$$
and $\vert\vert{\bf \Sigma}^{-1}\vert\vert_{1,1}\leq \vert\vert{\bf D}^{-1}\vert\vert_{1,1}\cdot\vert\vert{\bf T}^{-1}\vert\vert_{1,1}\cdot\vert\vert{\bf D}^{-1}\vert\vert_{1,1}\leq \frac{2C_7+C_8}{C_7C_8}(c_0\varepsilon^4)^{\frac{\gamma-2}{2}}$. By Lemma 5.1, $\vert\vert \textbf{x}\vert\vert_{l^1}$ can be bounded above by
\begin{equation*}
\vert\vert \textbf{x}\vert\vert_{l^1} \leq \sum_{k=-n_1+1}^{j-1} C_9c_0\varepsilon^4\left\vert (k-j)\varepsilon^2\right\vert^{-\gamma}.
\end{equation*}
As how we compute the upper bound for $\vert\vert{\bf A}\vert\vert_{1,1}$, we derive that 
\begin{equation*}
\begin{split}
\vert\vert{\bf \eta}\vert\vert_{l^1}&\leq \frac{2C_7+C_8}{C_7C_8}(c_0\varepsilon^4)^{\frac{\gamma-2}{2}}\vert\vert{\bf x}\vert\vert_{l^1}\leq\frac{2C_7+C_8}{C_8}\cdot\sum_{k=-n_1+1}^{j-1}\frac{C_9c_0\varepsilon^4\left\vert (k-j)\varepsilon^2\right\vert^{-\gamma}}{C_7(c_0\varepsilon^4)^{\frac{2-\gamma}{2}}}\\
&<\frac{2C_7+C_8}{C_8}\cdot\frac{1}{2}\cdot\frac{2C_7}{2C_7+C_8}=\frac{C_7}{C_8}.
\end{split}
\end{equation*}
\end{proof}

Finally, combining \eqref{sumprob}, \eqref{probB} and \eqref{probAbound} yields
\begin{equation*}
\begin{split}
P(F_{1})&\leq \frac{C_5}{(1\wedge \sqrt{{c_0}})\varepsilon^{2}}\exp\left(-\frac{C_6}{4\mathcal{D}^2t_1^{1-\gamma/2}}\right)+C\exp\left(-\frac{C'}{\varepsilon^2}\right)\\
&\leq C_5'\exp\left(-\frac{1}{2}\ln t_1-\frac{C_6}{4\mathcal{D}^2t_1^{1-\gamma/2}}\right)+C\exp\left(-\frac{C'}{\varepsilon^2}\right),
\end{split}
\end{equation*}
and we may choose a constant $\mathcal{D}_1(\gamma)>0$ such that for any $0<\mathcal{D}<\mathcal{D}_1$,
\begin{align*}
P(F_1)&\leq C_5'\exp\left(-\frac{C_6'}{\mathcal{D}^2t_1^{1-\gamma/2}}\right)+C\exp\left(-\frac{C'}{\varepsilon^2}\right)\\
&\leq C\exp\left(-\frac{C'}{\varepsilon^2+\mathcal{D}^2t_1^{1-\gamma/2}}\right).
\end{align*}
Since $\varepsilon^2>\mathcal{C}^{-\gamma/2}t_1^{\gamma/2}$ from \eqref{c0}, we finally attain
\begin{equation*}
P(F_1)\leq \textbf{C}_4\exp\left(-\frac{\textbf{C}_5}{t_1^{\gamma/2}}\right),
\end{equation*}
which completes the proof of \eqref{F1} as well as Proposition \ref{prop}(a).

\textbf{Proof of Proposition \ref{prop} (b)} We state the Gaussian correlation inequality, which is crucial to prove Proposition \ref{prop}.
\begin{lemma}\label{Gaussiancorr}For any convex symmetric sets $K, L$ in $\R$ and any centered Gaussian measure $\mu$ on $\R$, we have
$$
\mu(K\cap L)\geq \mu(K)\mu(L).
$$
\end{lemma}
\begin{proof}
See in paper \cite{royen2014simple}, \cite{latala2017royen}. 
\end{proof}
By the Markov property of $u(t,\cdot)$ (see page 247 in \cite{da2014stochastic}), the behavior of $u(t,\cdot)$ in the interval $[t_n,t_{n+1}]$ depends only on $u(t_n,\cdot)$ and $\dot{F}(t,x)$ on $[t_n,t]\times[-1,1]$. Therefore, as the proof of Proposition \ref{prop} (a), it is enough to show that
\begin{equation}
\label{E0}
P\left(E_{0}\right)\geq  {\bf C_6}\exp\left(-\frac{{\bf C_7}}{t_1^{1-\gamma/2}}\right),
\end{equation}
where ${\bf C_6}$ and ${\bf C_7}$ do not depend on $u_0$ and $\vert u_0(x)\vert\leq \frac{\mathcal{C}_3}{3}t_1^{\frac{2-\gamma}{4}}$ for every $x\in\mathbb{T}$. 

We start with the Gaussian case that $\sigma(t,x,u)=\sigma(t,x)$, which does not depend on $u$. And we assume that $\sigma$ is a smooth and deterministic function. For $n \geq 0$, we define a sequence of events  
$$D_n=\left\lbrace\vert u(t_{n+1},x)\vert\leq \frac{\mathcal{C}_3}{6}t_1^{(2-\gamma)/4}~\text{and}~\vert u(t,x)\vert\leq\frac{2\mathcal{C}_3}{3}t_1^{(2-\gamma)/4}~\forall t\in[t_n,t_{n+1}],x\in[-1,1]\right\rbrace.$$
Denote 
$$G_t(u_0)(x)=G_t*u_0(x)=\int_{-1}^1G_t(x-y)u_0(y)dy,$$
and we have
\begin{equation}
\label{u0convolute}
\vert G_t(u_0)(x)\vert\leq \sup_{x}\vert u_0(x)\vert\leq \frac{\mathcal{C}_3}{3}t_1^{\frac{2-\gamma}{4}}.
\end{equation}

We consider the smooth and bounded function $\frac{G_t(u_0)(x)}{t_1\sigma(t,x)}$ on $[0,t_1]\times\mathbb{T}$. \cite{roncal2016fractional} provided a pointwise formula for the fractional Laplacian of a certain class of functions on $\mathbb{T}$, so that one may assume that there is a continuous function $f(t,x)$ on $[0,t_1]\times\mathbb{T}$ such as
\begin{equation}
\label{PtwFracLap}
\begin{split}
f(t,x)&=(-\Delta_{\mathbb{T}})^{\frac{1-\gamma}{2}}\left(\frac{G_t(u_0)(x)}{t_1\sigma(t,x)}\right)\\
&=\int_{-1}^1\left(\frac{G_t(u_0)(x)}{t_1\sigma(t,x)}-\frac{G_t(u_0)(y)}{t_1\sigma(t,y)}\right)K_{1-\gamma}(x-y)dy\\
&\leq 2\sup_{s,x}\left\vert\frac{G_t(u_0)(x)}{t_1\sigma(s,x)}\right\vert\int_{-1}^1K_{1-\gamma}(x-y)dy\\
&\leq C\sup_{s,x}\left\vert\frac{G_t(u_0)(x)}{t_1\sigma(s,x)}\right\vert,
\end{split}
\end{equation}
where $K_{1-\gamma}$ is the positive kernel function on $\mathbb{T}$ defined in \cite{roncal2016fractional}, Theorem 1.5.

On the other hand, $f(t,x)$ is defined as 
\begin{equation*}
\begin{split}
f(t,x)&=(-\Delta_{\mathbb{T}})^{\frac{1-\gamma}{2}}\left(\frac{G_t(u_0)(x)}{t_1\sigma(t,x)}\right)\\
&=\sum_{m=0}^\infty \vert m\vert^{1-\gamma}C_m\left(\frac{G_t(u_0)(x)}{t_1\sigma(t,x)}\right)\exp(m\pi ix),
\end{split}
\end{equation*}
where $C_m\left(\frac{G_t(u_0)(x)}{t_1\sigma(t,x)}\right)$ is the $m^{\text{th}}$ Fourier coefficient. Thus, convolving $f$ with $\Lambda$ on spatial variable yields
\begin{equation}
\label{Rieszconv}
\int_{-1}^{1}f(s,z)\Lambda(y-z)dz=\frac{CG_s(u_0)(y)}{t_1\sigma(s,y)}.
\end{equation}

Now for a probability measure $Q$ given by
$$\frac{dQ}{dP}=\exp\left(Z_{t_1}-\frac{1}{2}\langle Z\rangle_{t_1}\right),$$
where
\begin{equation*}
Z_{t_1}=-\int_0^{t_1}\int_{-1}^1f(s,y)F(dyds),
\end{equation*}
if $Z_{t_1}$ satisfies Novikov's condition in \cite{allouba1998different}, then for each fixed $T\in[0,\infty)$, $0\leq t\leq T$, and $\forall A\in\mathcal{B}(\mathbb{T})$, 
\begin{equation*}
\widetilde{F}_t(A):=F_t(A)+\int_{[0,t]\times A^{2}}f(s,y)\Lambda(x-y)dxdyds
\end{equation*}
is a centered spatially homogeneous Wiener process under the measure $Q$ (see \cite{allouba1998different} for more details). Moreover, for $x\in \mathbb{T}$, the covariance structure of $\dot{F}(t,x)$ provides the formal form of $\dot{\widetilde{F}}(t,x)$ as follows,
\begin{align*}
\dot{\widetilde{F}}(t,x)=\dot{F}(t,x)+\int_{-1}^1f(t,y)\Lambda(x-y)dy,
\end{align*}
which is a spatially homogeneous noise under measure $Q$. 

We now check whether $Z_{t_1}$ satisfies Novikov's condition. When $f$ is a deterministic function defined in \eqref{PtwFracLap}, it is equivalent to verify that
\begin{equation}
\label{novi}
\int_0^{t_1}\int_{-1}^1\int_{-1}^1f(s,y)f(s,z)\Lambda(y-z)dzdyds<+\infty.
\end{equation}
Given \eqref{u0convolute}, \eqref{PtwFracLap} and \eqref{Rieszconv}, we derive
\begin{equation}
\label{novikov}
\begin{split}
&\int_0^{t_1}\int_{-1}^1\int_{-1}^1f(s,y)f(s,z)\Lambda(y-z)dzdyds=\int_0^{t_1}\int_{-1}^1f(s,y)\frac{G_t(u_0)(y)}{t_1\sigma(s,y)}dyds\\
&\leq\int_0^{t_1}\int_{-1}^1\vert f(s,y)\vert\left(\frac{\mathcal{C}_3}{3}t_1^{\frac{2-\gamma}{4}}\right)\cdot\frac{1}{t_1\mathcal{C}_1}dyds=\frac{C(\mathcal{C}_1,\mathcal{C}_3)}{t_1^{\frac{2+\gamma}{4}}}\int_0^{t_1}\int_{-1}^1\vert f(s,y)\vert dyds\\
&\leq\frac{C(\gamma,\mathcal{C}_1,\mathcal{C}_3)}{t_1^{\frac{2+\gamma}{4}}}\int_0^{t_1}\sup_{s,x}\left\vert\frac{G_t(u_0)(x)}{t_1\sigma(s,x)}\right\vert ds\leq C(\gamma,\mathcal{C}_1,\mathcal{C}_3)t_1^{-\gamma/2}<+\infty,
\end{split}
\end{equation}
which satisfies \eqref{novi}. Thus, we can rewrite equation $(1.1)$ with deterministic $\sigma$ as
\begin{align*}
u(t,x)&=G_t(u_0)(x)+\int_0^t\int_{-1}^1G_{t-s}(x-y)\sigma(s,y)\left[\widetilde{F}(dyds)-\frac{G_s(u_0)(y)}{t_1\sigma(s,y)}dyds\right]\\
&=G_t(u_0)(x)-\frac{G_t(u_0)(x)}{t_1}+\int_0^t\int_{-1}^1G_{t-s}(x-y)\sigma(s,y)\widetilde{F}(dyds)\\
&=\left(1-\frac{t}{t_1}\right)G_t(u_0)(x)+\int_0^t\int_{-1}^1G_{t-s}(x-y)\sigma(s,y)\widetilde{F}(dyds).
\end{align*}
The first term is $0$ at $t_1$, and since $\vert u_0(x)\vert\leq \frac{\mathcal{C}_3}{3}t_1^{\frac{2-\gamma}{4}}$, we have
$$\left|\left(1-\frac{t}{t_1}\right)G_t(u_0)(x)\right|\leq\frac{\mathcal{C}_3}{3}t_1^{\frac{2-\gamma}{4}}, \forall x\in[-1,1],t<t_1.$$
We define
$$\widetilde{N}(t,x)=\int_0^t\int_{-1}^1G_{t-s}(x-y)\sigma(s,y)\widetilde{F}(dyds).$$
$\widetilde{F}$ may not have the same covariance as $F$ under the probability measure $P$; however, under probability measure $Q$, we can apply Lemma \ref{larged} to $\widetilde{F}$ and derive
\begin{equation*}
\begin{split}
Q\left(\sup_{\substack{0\leq t\leq c_0\varepsilon^4\\ x\in\left[0,c_0\varepsilon^2\right]}}\left\vert \widetilde{N}(t,x)\right\vert>\frac{\mathcal{C}_3}{6}t_1^{(2-\gamma)/4}\right)&=Q\left(\sup_{\substack{0\leq t\leq (c_0)^{-1}(\sqrt{c_0}\varepsilon)^4\\ x\in\left[0,(\sqrt{c_0}\varepsilon)^2\right]}}\left\vert \widetilde{N}(t,x)\right\vert>\frac{\mathcal{C}_3}{6}t_1^{(2-\gamma)/4}\right)\\
&\leq C_5\exp\left(-\frac{\mathcal{C}_3^2C_6}{36\mathcal{C}_2^2}\right)<1,
\end{split}
\end{equation*}
where $\beta = c_0^{-1}>1$ and $\lambda = \frac{c_0^{(\gamma-2)/4}}{6}$ in Lemma 4.4. The last inequality is derived from the definition of $\mathcal{C}_3$ in \eqref{c3}. By the Gaussian correlation inequality (Lemma \ref{Gaussiancorr}), we obtain
\begin{align*}
Q\left(\sup_{\substack{0\leq t\leq t_1\\ x\in[-1,1]}}\left\vert \widetilde{N}(t,x)\right\vert\leq\frac{\mathcal{C}_3}{6}t_1^{(2-\gamma)/4}\right)&\geq Q\left(\sup_{\substack{0\leq t\leq t_1\\ x\in\left[0,c_0\varepsilon^2\right]}}\left\vert \widetilde{N}(t,x)\right\vert\leq\frac{\mathcal{C}_3}{6}t_1^{(2-\gamma)/4}\right)^{\frac{2}{c_0\varepsilon^2}}\\
&\geq \left[1-C_5\exp\left(-\frac{\mathcal{C}_3^2C_6}{36\mathcal{C}_2^2}\right)\right]^{\frac{2}{c_0\varepsilon^2}}.
\end{align*}
Observe that
$$Q(D_0)\geq Q\left(\sup_{\substack{0\leq t\leq t_1\\ x\in[-1,1]}}\left\vert \widetilde{N}(t,x)\right\vert\leq\frac{\mathcal{C}_3}{6}t_1^{(2-\gamma)/4}\right),$$
and if we replace $f(s,y)$ with $2f(s,y)$ for $Z_{t_1}$,
\begin{equation}
\label{radon}
1=\E\left[\frac{dQ}{dP}\right]=\E\left[\exp\left(Z_{t_1}-\frac{1}{2}\langle Z\rangle_{t_1}\right)\right]=\E[\exp\left(2Z_{t_1}-2\langle Z\rangle_{t_1}\right)].
\end{equation}
Because $f(s,y)$ is deterministic, we may estimate the Radon-Nikodym derivative as follow,
\begin{align*}
\E\left[\left(\frac{dQ}{dP}\right)^2\right]&=\E[\exp\left(2Z_{t_1}-\langle Z\rangle_{t_1}\right)]=\E[\exp\left(2Z_{t_1}-2\langle Z\rangle_{t_1}\right)\cdot\exp(\langle Z\rangle_{t_1})]\\
&\leq\exp\left(\frac{C}{t_1^{\gamma/2}\mathcal{C}_1^2}\right).
\end{align*}
The last inequality follows from \eqref{novikov} and \eqref{radon}. The Cauchy-Schwarz inequality implies
$$Q(D_0)\leq \sqrt{\E\left[\left(\frac{dQ}{dP}\right)^2\right]}\cdot\sqrt{P(D_0)},$$
and as a consequence, we get
\begin{equation}
\begin{split}
\label{probD}
P(D_0)&\geq \exp\left(-\frac{C}{t_1^{\gamma/2}\mathcal{C}_1^2}\right)\exp\left(\frac{4}{c_0\varepsilon^2}\ln \left[1-C_5\exp\left(-\frac{\mathcal{C}_3^2C_6}{36\mathcal{C}_2^2}\right)\right]\right)\\
&\geq \exp\left(-\frac{C}{t_1^{\gamma/2}\mathcal{C}_1^2}-\frac{C'}{c_0\varepsilon^2}\right)\\
&\geq C\exp\left(-\frac{C'}{t_1^{1-\gamma/2}}\right).
\end{split}
\end{equation}
Recall from \eqref{c0} that $c_0\varepsilon^2 =\mathcal{C}_0^{\gamma/2}t_1^{1-\gamma/2}$, so the term $\frac{C'}{c_0\varepsilon^2}$ will dominate the term $\frac{C}{t_1^{\gamma/2}}$ as $\varepsilon$ approaching 0, which leads to the last inequality.

For the lower bound with a non-deterministic $\sigma(t,x,u)$, we write
$$u(t,x)=u_g(t,x)+M(t,x),$$
where $u_g(t,x)$ satisfies the equation
$$\partial_tu_g=\partial^2_xu_g+\sigma(t,x,u_0(x))\dot{F}(t,x),$$
and 
$$M(t,x)=\int_0^t\int_{-1}^1G_{t-s}(x-y)[\sigma(s,y,u(s,y))-\sigma(s,y,u_0(y))]F(dyds)$$
with an initial profile $u_0$. Since $u_g$ is Gaussian, for an event defined as
$$\widetilde{D}_0=\left\lbrace\vert u_g(t_{1},x)\vert\leq \frac{\mathcal{C}_3}{6}t_1^{(2-\gamma)/4}~\text{and}~\vert u_g(t,x)\vert\leq\frac{2\mathcal{C}_3}{3}t_1^{(2-\gamma)/4}~\forall t\in[0,t_{1}],x\in[-1,1]\right\rbrace,$$
we can apply \eqref{probD} to it and get
\begin{equation}
\label{probD0}
P\left(\widetilde{D}_0\right)\geq C\exp\left(-\frac{C'}{t_1^{1-\gamma/2}}\right).
\end{equation}

Define the stopping time 
$$\tau=\inf\left\lbrace t>0:\vert u(t,x)-u_0(x)\vert>2\mathcal{C}_3t_1^{(2-\gamma)/4}\text{~for some $x\in[-1,1]$}\right\rbrace,$$
and if the set is empty, we set $\tau=\infty$. Clearly we have $\tau>t_1$ on the event $E_{0}$ in \eqref{En} since $\vert u_0(x)\vert\leq \frac{\mathcal{C}_3}{3}t_1^{\frac{2-\gamma}{4}}$, and $\vert u(t,x)\vert\leq \mathcal{C}_3t_1^{(2-\gamma)/4}$ for $\forall t\in[0,t_1]$ on the event $E_{0}$. We make another definition
$$\widetilde{M}(t,x)=\int_0^t\int_{-1}^1G_{t-s}(x-y)[\sigma(s,y,u(s\wedge \tau,y))-\sigma(s,y,u_0(y))]F(dyds),$$
and $M(t,x)=\widetilde{M}(t,x)$ for $t\leq t_1$ on the event $\{\tau>t_1\}$. Moreover, we have
\begin{equation}
\label{E0bound}
\begin{split}
P(E_{0})&\geq P\left(\widetilde{D}_0\bigcap \left\lbrace \sup_{\substack{0\leq t\leq t_1\\ x\in[-1,1]}}\vert M(t,x)\vert\leq\frac{\mathcal{C}_3}{6}t_1^{\frac{2-\gamma}{4}}\right\rbrace\right)\\
&=P\left(\left(\widetilde{D}_0\bigcap \left\lbrace \sup_{\substack{0\leq t\leq t_1\\ x\in[-1,1]}}\vert M(t,x)\vert\leq\frac{\mathcal{C}_3}{6}t_1^{\frac{2-\gamma}{4}}\right\rbrace\bigcap\{\tau>t_1\}\right)\right.\\
&\hspace{3cm}\bigcup\left.\left(\widetilde{D}_0\bigcap \left\lbrace \sup_{\substack{0\leq t\leq t_1\\ x\in[-1,1]}}\vert M(t,x)\vert\leq\frac{\mathcal{C}_3}{6}t_1^{\frac{2-\gamma}{4}}\right\rbrace\bigcap\{\tau\leq t_1\}\right)\right).
\end{split}
\end{equation}
On the event $\{\tau>t_1\}$, we have
$$\sup_{\substack{0\leq t\leq t_1\\ x\in[-1,1]}}\vert M(t,x)\vert=\sup_{\substack{0\leq t\leq t_1\\ x\in[-1,1]}}\vert \widetilde{M}(t,x)\vert,$$
and on the event $\widetilde{D}_0\cap\{\tau\leq t_1\}$,
$$\vert u_g(\tau,x)\vert\leq \frac{2\mathcal{C}_3}{3}t_1^{\frac{2-\gamma}{4}}\quad\text{and}\quad \vert u_0(x)\vert\leq\frac{\mathcal{C}_3}{3}t_1^{\frac{2-\gamma}{4}}\quad\text{for all}\quad x\in\mathbb{T},$$
$$\vert u(\tau,x)-u_0(x)\vert\geq 2\mathcal{C}_3t_1^{\frac{2-\gamma}{4}}\quad\text{for some}\quad x\in\mathbb{T}.$$
The above inequalities lead to
\begin{align*}
\sup_{\substack{x}}\vert M(\tau,x)\vert&=\sup_{\substack{x}}\vert u(\tau,x)-u_g(\tau,x)\vert\geq \sup_{\substack{x}}(\vert u(\tau,x)\vert-\vert u_g(\tau,x)\vert)\\
&\geq \sup_{\substack{x}}\vert u(\tau,x)\vert-\frac{2\mathcal{C}_3}{3}t_1^{\frac{2-\gamma}{4}}\geq 2\mathcal{C}_3t_1^{\frac{2-\gamma}{4}}-\frac{\mathcal{C}_3}{3}t_1^{\frac{2-\gamma}{4}}-\frac{2\mathcal{C}_3}{3}t_1^{\frac{2-\gamma}{4}}\\
&> \frac{\mathcal{C}_3}{6}t_1^{\frac{2-\gamma}{4}},
\end{align*}
which implies
$$\widetilde{D}_0\cap \left\lbrace \sup_{\substack{0\leq t\leq t_1\\ x\in[-1,1]}}\vert M(t,x)\vert\leq\frac{\mathcal{C}_3}{6}t_1^{\frac{2-\gamma}{4}}\right\rbrace\cap\{\tau\leq t_1\}=\phi.$$

Combining the above with \eqref{E0bound} yields
\begin{equation}
\label{E0boundsim}
\begin{split}
P(E_{0})&\geq P\left(\widetilde{D}_0\cap \left\lbrace \sup_{\substack{0\leq t\leq t_1\\ x\in[-1,1]}}\vert \widetilde{M}(t,x)\vert\leq\frac{\mathcal{C}_3}{6}t_1^{\frac{2-\gamma}{4}}\right\rbrace\right)\\
&\geq P\left(\widetilde{D}_0\right)-P\left( \sup_{\substack{0\leq t\leq t_1\\ x\in[-1,1]}}\vert \widetilde{M}(t,x)\vert>\frac{\mathcal{C}_3}{6}t_1^{\frac{2-\gamma}{4}}\right),
\end{split}
\end{equation}
and as a union bound calculated in the equation \eqref{probB}, we obtain
\begin{equation}
\label{probDtilde}
P\left( \sup_{\substack{0\leq t\leq t_1\\ x\in[-1,1]}}\left\vert \widetilde{M}(t,x)\right\vert>\frac{\mathcal{C}_3}{6}t_1^{\frac{2-\gamma}{4}}\right)\leq \frac{C_5}{(1\wedge\sqrt{c_0})\varepsilon^2}\exp\left(-\frac{C_6\mathcal{C}_3^2}{144\mathcal{D}^2 t_1^{1-\gamma/2}}\right).
\end{equation}
Consequently, from \eqref{probD0}, \eqref{E0boundsim} and \eqref{probDtilde}, we conclude that there exists a constant $\mathcal{D}_2(\gamma)>0$ such that for any $\mathcal{D}<\mathcal{D}_2$,
\begin{equation}
\begin{split}
\label{E_0}
P(E_{0})&\geq C\exp\left(-\frac{C'}{t_1^{1-\gamma/2}}\right)-C_5'\exp\left(-\frac{C_6'}{\mathcal{D}^2t_1^{1-\gamma/2}}\right)\\
&\geq \textbf{C}_6\exp\left(-\frac{\textbf{C}_7}{t_1^{1-\gamma/2}}\right),
\end{split}
\end{equation}
which completes the proof of \eqref{E0} as well as Proposition \ref{prop}(b).

\section{Acknowledgment}
The author would like to thank his advisor Professor Carl Mueller and the anonymous referees for constructive comments.

\bibliographystyle{alpha}
\bibliography{mybib}

\newcommand{\etalchar}[1]{$^{#1}$}
\begin{thebibliography}{DKM{\etalchar{+}}09}

\bibitem[AJM21]{athreya2021small}
Siva Athreya, Mathew Joseph, and Carl Mueller.
\newblock Small ball probabilities and a support theorem for the stochastic
  heat equation.
\newblock {\em The Annals of Probability}, 49(5):2548--2572, 2021.

\bibitem[All98]{allouba1998different}
Hassan Allouba.
\newblock Different types of spdes in the eyes of girsanov's theorem.
\newblock {\em Stochastic analysis and applications}, 16(5):787--810, 1998.

\bibitem[Bez16]{bezdek2016weak}
Pavel Bezdek.
\newblock On weak convergence of stochastic heat equation with colored noise.
\newblock {\em Stochastic Processes and their Applications}, 126(9):2860--2875,
  2016.

\bibitem[Bro83]{brosamler1983laws}
GA~Brosamler.
\newblock Laws of the iterated logarithm for brownian motions on compact
  manifolds.
\newblock {\em Zeitschrift f{\"u}r Wahrscheinlichkeitstheorie und Verwandte
  Gebiete}, 65(1):99--114, 1983.

\bibitem[Dal99]{dalang1999extending}
Robert Dalang.
\newblock Extending the martingale measure stochastic integral with
  applications to spatially homogeneous spde's.
\newblock {\em Electronic Journal of Probability}, 4:1--29, 1999.

\bibitem[DKM{\etalchar{+}}09]{dalang2009minicourse}
Robert~C Dalang, Davar Khoshnevisan, Carl Mueller, David Nualart, and Yimin
  Xiao.
\newblock {\em A minicourse on stochastic partial differential equations},
  volume 1962.
\newblock Springer, 2009.

\bibitem[DPZ14]{da2014stochastic}
Giuseppe Da~Prato and Jerzy Zabczyk.
\newblock {\em Stochastic equations in infinite dimensions}.
\newblock Cambridge university press, 2014.

\bibitem[FJK22]{foondun2022small}
Mohammud Foondun, Mathew Joseph, and Kunwoo Kim.
\newblock Small ball probability estimates for the h{\"o}lder semi-norm of the
  stochastic heat equation.
\newblock {\em Probability Theory and Related Fields}, pages 1--61, 2022.

\bibitem[LM17]{latala2017royen}
Rafa{\l} Lata{\l}a and Dariusz Matlak.
\newblock Royen’s proof of the gaussian correlation inequality.
\newblock In {\em Geometric aspects of functional analysis}, pages 265--275.
  Springer, 2017.

\bibitem[LS01]{li2001gaussian}
Wenbo~V Li and Q-M Shao.
\newblock Gaussian processes: inequalities, small ball probabilities and
  applications.
\newblock {\em Handbook of Statistics}, 19:533--597, 2001.

\bibitem[Roy14]{royen2014simple}
Thomas Royen.
\newblock A simple proof of the {Gaussian} correlation conjecture extended to
  some multivariate gamma distributions.
\newblock {\em Far East J. Theor. Stat.}, 48(2):139--145, 2014.

\bibitem[RS16]{roncal2016fractional}
Luz Roncal and Pablo~Ra{\'u}l Stinga.
\newblock Fractional laplacian on the torus.
\newblock {\em Communications in Contemporary Mathematics}, 18(03):1550033,
  2016.

\bibitem[SSS02]{sanz2002holder}
Marta Sanz-Sol{\'e} and Monica Sarr{\`a}.
\newblock H{\"o}lder continuity for the stochastic heat equation with spatially
  correlated noise.
\newblock In {\em Seminar on Stochastic Analysis, Random Fields and
  Applications III}, pages 259--268. Springer, 2002.

\bibitem[TTV03]{tindel2003stochastic}
S~Tindel, CA~Tudor, and F~Viens.
\newblock Stochastic evolution equations with fractional brownian motion.
\newblock {\em Probability Theory and Related Fields}, 127(2):186--204, 2003.

\bibitem[TV99]{tindel1999space}
Samy Tindel and Frederi Viens.
\newblock On space-time regularity for the stochastic heat equation on lie
  groups.
\newblock {\em Journal of Functional Analysis}, 169(2):559--603, 1999.

\bibitem[Wal86]{walsh1986anintroductiontostochastic}
JB~Walsh.
\newblock An introduction to stochastic partial differential equations.
\newblock {\em Ecole d'Ete de Probabilites de Saint Flour XIV-1984, Lecture
  Notes in Math}, 1180, 1986.

\end{thebibliography}

\end{document}